\documentclass[12pt]{article}

\RequirePackage{kpfonts}
\RequirePackage[sf,bf,small,raggedright]{titlesec}
\usepackage{graphicx} 
\usepackage{amsmath,amssymb}
\usepackage{tikz} 
\usepackage{hyperref}
\usepackage{lipsum}
\usepackage{setspace}
\usepackage{bbm}
\usepackage{tikz}
\usetikzlibrary{shapes.geometric}
\usepackage{microtype}

\usetikzlibrary{shapes}
\usetikzlibrary{shapes,fit}
\usetikzlibrary{shapes.misc}
\usetikzlibrary{trees, shapes}
\usetikzlibrary{shapes, positioning}

\usepackage{amsthm}
\usepackage{diagbox}
\usepackage{dsfont}

\usepackage{listings}
\usepackage{bm}
\newtheorem{theorem}{Theorem}
\newtheorem{definition}[theorem]{Definition}

\newtheorem{corollary}[theorem]{Corollary}
\newtheorem{remark}{Remark}

\usepackage{graphicx} 
\newtheorem{lemma}[theorem]{Lemma}
\usepackage[backend=biber,style=numeric,sorting=nyt,maxnames=10]{biblatex} 
\usepackage{amsfonts} 
\usepackage{fancyhdr}

\newcommand{\Thref}[1]{\hyperref[#1]{\textbf{Th.~\ref*{#1}}}}
\newcommand{\corref}[1]{\hyperref[#1]{\textbf{Co.~\ref*{#1}}}}

\newcommand{\N}{\mathbb{N}}

\numberwithin{equation}{section}

\newcommand{\EXP}{\mathbb{E}}
\newcommand{\PROB}{\mathbb{P}}

\newcommand{\defeq}{\stackrel{\mathrm{def.}}{=}}

\begin{document}
\title{\bf{A study of centrality measures in random recursive trees}
 \thanks{Luc Devroye was supported by \textsc{nserc} grant \textsc{rgpin}-2024-04164. Richard Coll Josifov acknowledges the support of the Spanish Ministry of Science, Innovation and Universities through grant FPU23/01138. Gábor Lugosi acknowledges the support of the Spanish Ministry of Economy and Competitiveness grant PID2022138268NB-I00, financed by MCIN/AEI/10.13039/501100011033, FSE+MTM2015-67304-P, and FEDER, EU}

\author{Richard Coll Josifov \\ [-0.1em]
\small Department of Mathematics,\\ [-0.3em] \small Polytechnic University of Catalonia, Barcelona, Spain \\ [-0.3em] \small Department of Economics and Business,\\ [-0.3em] \small Pompeu Fabra University, Barcelona, Spain 
\and
  Luc Devroye \\ [-0.1em]
\small
  School of Computer Science, McGill University, \\ [-0.3em]
\small
  Montreal, Canada 
\and
G\'abor Lugosi \\ [-0.1em]
\small
Department of Economics and Business, \\ [-0.3em]
\small
Pompeu Fabra University, Barcelona, Spain \\ [-0.3em]
\small
ICREA, Pg. Lluís Companys 23, 08010 Barcelona, Spain \\ [-0.3em]
\small
Barcelona Graduate School of Economics
 }
}
\maketitle
\vspace{-2em}

\begin{abstract}
    We investigate the behaviour of five classical centrality measures—Jordan, rumor, betweenness, degree, and closeness centralities---in the setting of uniform random recursive trees. Motivated by applications in network archaeology, we focus on two fundamental questions: (i) the birth index (time of arrival) of the most central vertex, and (ii) the relative centrality of the root.

We quantify the probability that the root is the most central vertex, analyze its expected rank under each centrality measure, and determine the expected birth index of a central vertex. In addition, we characterize the typical size of the set of top-ranked vertices that contains the root with high probability. Finally, for each centrality notion, we study the persistence properties of the center and the asymptotic behaviour of the root’s rank.
\end{abstract}


\onehalfspacing

\footnotesize
\tableofcontents
\normalsize

\section{Introduction}

Centrality is a fundamental concept in network science, capturing the structural importance of vertices in complex networks. Over the years, a wide range of centrality measures has been proposed to quantify different aspects of vertex importance. For a comprehensive survey of centrality measures, see Saxena and Iyengar \cite{centralitiessurvey}.

This paper aims to provide a rigorous analysis of five prominent centrality measures in one of the most basic and canonical models of network growth: the \emph{uniform random recursive tree}. The measures we study are Jordan centrality, betweenness centrality, degree centrality, rumor centrality, and closeness centrality. Precise definitions of these notions are given below. Each of these measures induces a ranking of the vertices from most to least central; we refer to a vertex of maximum centrality as a \emph{center} of the network.

Our investigation is primarily motivated by questions arising in \emph{network archaeology} \cite{networkArchaeology}, a branch of combinatorial statistics concerned with inferring past properties of growing networks from present-day observations. A central problem in this area is \emph{root estimation}: identifying the initial vertex, or root, of a randomly grown network. This problem has been studied extensively for tree-valued models, including uniform attachment trees, preferential attachment trees, and related growth processes.

In this paper, we focus on the \emph{uniform random recursive tree} (\textsc{urrt}), also known as the \emph{uniform attachment tree}. This model generates a sequence of random labelled increasing trees ${T_n}$, where for each $n \ge 1$, the vertex set of $T_n$ is $V(T_n) = [n]$. The construction proceeds recursively: starting from a single vertex $1$, designated as the root, the tree $T_{n+1}$ is obtained from $T_n$ by adding a new vertex $n+1$ and connecting it by an edge ${i, n+1}$, where $i$ is chosen uniformly at random from $[n]$.

To define the centrality measures considered in the paper, we introduce some notation. Given a tree $T_n$ and  a vertex $v\in V(T_n)$, we write $(T_n,v)$ for the tree $T_n$ rooted at vertex $v$.
If $u\in V(T_n)$ is another vertex, then we write $(T_n,v)_{u\downarrow}$ for the subtree of $T_n$, rooted at vertex $u$, containing all vertices $w$ such that the path from $w$ to $v$ contains $u$.

    Given a tree $T_n$, and a vertex $v\in V(T_n)$, we define the following centrality measures:
    \begin{itemize}
    \renewcommand{\labelenumi}{(\roman{enumi})}
        \item \textsc{Jordan centrality}, first introduced by Jordan \cite{jordan}, ranks vertices according to the size of the largest component that results when the vertex is removed. Formally, vertices are ranked according to increasing values of the centrality measure $\psi_{T_n}(v) = \max_{u \in V(T_n)  \backslash v} |(T_n,v)_{u\downarrow}|$.
        \item \textsc{Betweenness centrality}, introduced by Freeman \cite{betweennessFreeman}, ranks vertices $v$ according to the number of vertex pairs for which the unique path between them passes through $v$.
More precisely, vertices are ranked according to decreasing values of the function
$\mathcal{B}_T(v) = \sum_{s,t \in V(T_n) : s,t \neq v}\mathbbm{1}_{v \in\sigma_{s,t}}$, where $\sigma_{s,t}$ is the path between vertices $s,t$ in $T_n$ excluding $s$ and $t$. 
        \item \textsc{Degree centrality} ranks vertices according to the degree $\deg_{T_n}(v)$ of vertex $v$ in $T_n$:
 $ d_{T_n}(v) = \deg_{T_n}(v)$ such that higher-degree vertices are more central.
        \item \textsc{Rumor centrality} was introduced by  Shah and Zaman  \cite{rumor} as the likelihood that a vertex is the source of a certain diffusion process on infinite regular trees. It quantifies the centrality of a vertex as the product
 $\varphi_{T_n}(v)= \prod_{u\in V(T_n) \backslash v}|(T_n,v)_{u\downarrow}|$. Smaller values of $\varphi_{T_n}(v)$ correspond to more central vertices.
        \item \textsc{Closeness centrality}, formally defined by Freeman \cite{freemanCloseness}, measures the centrality of a vertex by the sum of its distances to all other vertices: $\mathcal{C}_{T_n}(v)= \sum_{u \in V(T_n)  \backslash v} \text{dist}(v,u)$ where $\text{dist}(v,u)$ is the length of the path between $v$ and $u$ in $T_n$. Again, the smaller $\mathcal{C}_{T_n}(v)$, the more central $v$ is.
    \end{itemize}

 \begin{remark}
   For Jordan, rumor, and closeness centralities, smaller values of the centrality score correspond to higher centrality. In contrast, for degree and betweenness centralities, larger values indicate that a vertex is more central.
\end{remark}

\begin{remark}
The ranking of vertices induced by betweenness centrality $\mathcal{B}_T(v)$ is identical to the ranking produced by the measure:
\[ 
B'_T(v) = \sum_{u \in T_n : \{u,v\} \in E(T_n)} |(T_n,v)_{u\downarrow}|^2, 
\]
where a vertex is considered more central if $B'_T(v)$ is minimized. This measure can be viewed as a modification of Jordan centrality; rather than considering the size of the largest subtree, it computes the sum of squares of all subtree sizes.

This notion generalizes via the centrality measure $B^q_T(v)$ for a parameter $q \ge 2$:
\[ 
B^q_T(v) = \sum_{u \in T_n : \{u,v\} \in E(T_n)} |(T_n,v)_{u\downarrow}|^q. 
\]
The parameter $q$ allows for interpolation between betweenness centrality ($q=2$) and Jordan centrality ($q \to \infty$). While most results presented in this paper for betweenness centrality extend to general values of $q$, we focus on $q=2$ for the sake of readability.
\end{remark}

\subsection{Rank of the root and index of the center}

For any given centrality measure, one may order the vertices from most to least central.
In this paper, we focus on the behavior of two random variables: the \emph{rank of the root}
and the \emph{index of the center}. The rank of the root is the position of the root vertex in the
ranking. This is an important quantity for root finding methods: if the rank of the root is guaranteed to be
small with high probability, then the set of a few of the most central vertices contains the root with
high probability, giving a small \emph{confidence set} for root finding. An equally interesting---albeit less studied---notion for network archaeology
is the index (i.e., label) of the most central vertex. If the index of the center is small with high probability, then  
the centrality measure is effective in finding an early-arriving vertex.

To define these random variables rigorously, consider a centrality measure 
$\psi_{T_n}: V(T_n) \to \mathbb{N}$.
This defines a ranking of the vertices $\gamma_n:V(T_n) \to \{1,\ldots,n\}$ such that $\gamma_n(u) = 1 + |\{v\in T_n : v \text{ is more central than } u\}|$. We recall that vertices with the same centrality value have the same rank. 

When the centrality measure $\psi_{T_n}$ assigns equal values to different vertices, such vertices have the same
rank according to the definition above. To avoid technicalities arising from such ties,
we break ties so that each vertex is guaranteed to have a unique rank. Even though the tie-breaking
method has no consequence to our results, for concreteness, we consider the \emph{pessimistic} tie-breaking
rule. This rule orders vertices with equal centrality in reverse order of their insertion time.
Hence, if $\psi_{T_n}(u)=\psi_{T_n}(v)$ for some $u,v\in [n]$,
then $\gamma_n(u) < \gamma_n(v)$ if and only if
$u>v$.  Note that this rule ensures that $\left(\gamma_n(1),\ldots \gamma_n(n)\right)$ is a permutation of $[n]$.
In the rest of the paper, we implicitly assume this tie-breaking rule.

\begin{definition}
Given a centrality measure $\psi_{T_n}$, the \emph{center}  (or $\psi_{T_n}$-center) is the vertex $v\in [n]$ for which
$\gamma_n(v) =1$.

The \emph{index of the center} (i.e., the value of the integer $v$) is denoted by $I_n$.
Note that, thanks to tie-breaking, the center is always unique and well-defined.

The \emph{rank of the root} is the random variable $R_n=\gamma_n(1)$.
\end{definition}

\begin{remark}
  A standard notion is the \emph{centroid} of a tree $T_n$, defined as a vertex whose removal decomposes
  the tree into a forest whose largest component has size at most $n/2$. It is well known that
  every tree has at least one and at most two centroids. It is easy to see that a centroid is most central
  according to Jordan, closeness, and rumor centralities (see Section \ref{sec:prelim} for details).
  In other words, the Jordan center, closeness center, and rumor center coincide. In particular,
  the random variable $I_n$ is the same for these three centrality measures. At the same time, $R_n$
  behaves quite differently in these three cases.
\end{remark}

In this paper, we study various properties of the random variables $I_n$ and $R_n$.
In particular, we investigate the following quantities and properties:
\begin{itemize}
\item
  The probability $\PROB(R_n=1)=\PROB(I_n=1)$ that the root is the center. It was known since Moon \cite{MoonJordan} that the centroid (i.e., the Jordan, closeness, and rumor centers) is the root vertex with probability
  bounded away from zero. We show that this also holds for the betweenness center.
\item
  The expected values $\EXP R_n$ and $\EXP I_n$. We prove that in the case of rumor centrality, both
  quantities are of constant order. Interestingly, in all other cases, $\EXP I_n$ is of a smaller order of magnitude
  than $\EXP R_n$, indicating that, even though some centrality measures fail to assign a small rank to the root vertex,
  the most central vertex has a small expected index.
\item
  We say that a centrality measure is \emph{good for root estimation} if the sequence of random variables
  $\{R_n\}$ is tight, that is,  for every $\epsilon>0$ there
  exists a positive integer $K$ such that for all $n$, $\PROB\{R_n \le K\} \ge 1-\epsilon$.
  If a centrality measure is good for root estimation, then the set of vertices $\{v\in V(T_n): \gamma_n(v) \le K\}$
  is a \emph{confidence set} that contains the root with probability at least $1-\epsilon$ and has bounded size,
  independently of the size $n$ of the tree.  We show that Jordan, rumor, and betweenness centralities are good
  for root estimation, while closeness and degree centralities are not.
  When a centrality measure is good for root estimation, understanding the magnitude of $K$ that guarantees $\PROB\{R_n \le K\} \ge 1-\epsilon$ offers a way of comparing such centrality measures. For each of these centrality measures, we establish matching upper and lower bounds for the order of magnitude of $K$, as a function of $\epsilon$.
  Analogously, we may study the tightness of the sequence $\{I_n\}$. It follows from the boundedness of $\EXP I_n$ that
  the index of the center is tight for all centrality measures, with the only exception of degree centrality.
\item
  We also study \emph{persistence of the center }and \emph{persistence of the rank of the root}.
  Persistence means that $R_n$ and $I_n$ stabilize eventually, almost surely.
  Formally, for a given centrality measure, we say that the center is persistent if there exist integer-valued random
  variables $N,K$ with $\PROB\{N<\infty\}=\PROB\{K<\infty\}=1$, such that $I_n = K$ for all $n\ge N$.
  Similarly, the rank of the root is persistent if there exist integer-valued random
  variables $N,K$ with $\PROB\{N<\infty\}=\PROB\{K<\infty\}=1$, such that $R_n = K$ for all $n\ge N$.
  For each of the five centrality measures, we determine whether $I_n$ and $R_n$ are persistent or not.
\end{itemize}

The results of the paper are summarized in the tables of Section \ref{sec:results}. Naturally, not all of
these results are new. Previous work is surveyed in Section \ref{sec:related}.

\subsection{Related work}
\label{sec:related}


The study of root estimation in uniform random recursive trees (\textsc{urrt}) was pioneered by Haigh \cite{firstpapernetworkarcheology}, who demonstrated that the maximum likelihood estimator (\textsc{mle}) of the root is the tree's centroid. Haigh further showed that the probability of the \textsc{mle} correctly identifying the root converges to $1-\ln 2$ as the tree size $n \to \infty$. A parallel result was established by Shah and Zaman \cite{rumor} for rumor diffusion processes on infinite regular trees.

Building on these foundations, Bubeck et al. \cite{Bubeck2017} introduced the concept of \emph{confidence sets}---sets of vertices guaranteed to contain the root with high probability. They proved that for any $\epsilon > 0$, there exists a confidence set whose size depends solely on $\epsilon$ (independent of $n$) that contains the root with probability at least $1-\epsilon$. Crucially, they showed that such sets can be constructed using the most central vertices according to Jordan or rumor centralities—two of the five measures analyzed in the present work.

Recent advancements have further refined these bounds. Addario-Berry et al. \cite{addario2025leaf, addarioRumour} utilized a ``leaf-stripping'' method to provide bounded-size confidence sets and obtained upper bounds of the optimal order for both uniform attachment trees and diffusions in regular trees. For preferential attachment models, optimal upper bounds were established by Contat et al. \cite{contatEveAdam}.

\medskip\noindent\textsc{Properties of the centroid and betweenness.}
The statistical properties of the centroid were first investigated by Moon \cite{MoonJordan}, who derived the limiting probability $\lim_{n \to \infty} \PROB(I_n=1) = 1-\ln 2$ and provided exact formulas for the expected index $\EXP I_n$ and the expected distance between the centroid and the root. These results were later extended by Durant and Wagner \cite{durant2019centroid}, who obtained distributional results for the index and depth of the centroid in ``very simple increasing trees.'' In a separate comprehensive study, Durant and Wagner \cite{distributionBetweenness} explored betweenness centrality in \textsc{urrt}s and simply generated trees, characterizing its $k$-th moments, limiting distributions, and the convergence behavior of the betweenness center.

\medskip\noindent\textsc{Center coincidence and persistence.}
The relationship between different centrality measures is also well-documented; for instance, Zelinka \cite{propertiesCloseness} proved that Jordan and closeness centers coincide in any tree. A significant body of work has focused on the \emph{persistence} of these centers—the property that the ranking of the most central nodes stabilizes as the tree grows. Jog and Loh \cite{persistanceJogLoh, sublinearPA} and Banerjee and Bhamidi \cite{BanerjeeBhamidi2022} established strong persistence for Jordan centers across various models, including uniform, sublinear, and general attachment trees. Similarly, the persistence of the degree center has been studied for preferential attachment models \cite{Dereich2009, Galashin2016, BanerjeeBhamidiHubs, banerjee2023degree, contatEveAdam} and Crump-Mode-Jagers branching processes \cite{tyerCMJPersistance}.

\medskip\noindent\textsc{Alternative models and variations.}
Beyond simple trees, root estimation has been explored from various vantage points:
\begin{itemize}
    \item \textsc{Complex Networks:} Galton-Watson trees \cite{brandenberger2022root}, random recursive \textsc{dag}s and Cooper-Frieze networks \cite{briend2022archaeology}, and random nearest neighbor trees \cite{randomnearestneighbour}.
    \item \textsc{Computational Complexity:} Work by Borgs et al. \cite{borgs2012local} and Frieze and Pegden \cite{frieze_pegden_2017} shifted the focus toward the algorithmic efficiency of root identification.
    \item \textsc{Arrival Order:} Rather than identifying only the root, Crane and Xu \cite{crane2023} and Briend et al. \cite{briendhistory} developed methods to estimate the full arrival order of vertices in attachment models.
    \item \textsc{Seeded Models:} Investigations into finding the initial ``seed'' tree in evolving \textsc{urrt}s can be found in \cite{curien2015scaling, Bubeck2017, Reddad2019, lugosi2021finding}.
\end{itemize}

\subsection{Notation}\label{notation}
We write $N_n(v)=\{u\in V(T_n) \backslash v: \{u,v\} \in E(T_n)\}$ for the neighborhood of a vertex and write $T\backslash v$ to denote the forest obtained by removing vertex $v$ from $T$. $V(T)\backslash v$ denotes the vertex set after removing vertex $v$. 

Throughout this paper, we label the vertices in two ways. One is based on their arrival times in the tree (the root is vertex $1$, the next vertex is $2$, and so on). The other is based on an embedding in an \emph{Ulam-Harris tree} $\mathcal{U}$ (see \cite{Neveu}). This tree is the infinite ordered rooted tree with node set $\cup_{j=0}^{\infty}\mathbb{N}^j$, with the root as $\emptyset$ and vertices on level $j$ as $n_1\ldots n_j$ and edges from $n_1\ldots n_j$ to $n_1\ldots n_jn$ for all $j\ge 0$ and $n\in\mathbb{N}$. We define the \emph{zone} of a vertex $v=n_1\ldots n_j$, $v\neq \emptyset$ as $z(v)=\sum_{i=1}^j n_i$. 

We write $X\sim D$ to denote that the random variable $X$ is distributed as $D$.

\subsection{Results}
\label{sec:results}
The following tables summarize the known and new results. The first table presents the results for the rank of the root, whereas the second presents the results for the center's index.

\begin{table}[ht!]
    \centering
    \small
    \begin{tabular}{|c|c|c|c|c|}
    \hline
        &&&& \\
          & \begin{tabular}{@{}c@{}}$\displaystyle\lim_{n\to\infty}\PROB(R_n=1)$ \\$= \displaystyle \lim_{n\to\infty}\PROB(I_n=1)$\end{tabular}  & $\EXP R_n$  & \begin{tabular}{@{}c@{}}Good for root \\ estimation\end{tabular}  & \begin{tabular}{@{}c@{}}Persistence \\ of rank of\\ root\end{tabular}   \\ 
          &&&& \\
          \hline
          &&&& \\
         Jordan &  $1-\log(2)$ (\cite{firstpapernetworkarcheology}, \cite{MoonJordan})  & $\Theta (\log(n))$ (Cor. \ref{JordanExpected}) & $\checkmark$ (\cite{findingadam}) & yes (Th. \ref{persistencerankrootJordan}) \\ 
         &&&& \\
         \hline
         &&&& \\
         closeness & $1-\log(2)$ (\cite{firstpapernetworkarcheology},\cite{MoonJordan})  &$ e^{\Omega\left(\log^{1 \over 3} (n)\right)}$ (Th. \ref{ClosenessExpected}) & $\times$ (Th. \ref{ClosenessRootEstimation})  &  no (Th. \ref{ClosenessPersistence}) \\ 
         &&&& \\
         \hline
         &&&& \\
          rumor & $1-\log(2)$ (\cite{firstpapernetworkarcheology},\cite{MoonJordan}) & $\mathcal{O} (1)$ (\cite{findingadam}) & $\checkmark$ (\cite{findingadam}) &  yes (Th. \ref{rumorPersistence}) \\ 
          &&&& \\
         \hline
         &&&& \\
         betweenness & $> 0.07$ (Th. \ref{probabilityBetweenness}) & $\Theta (\log(n))$ (Th. \ref{BetweennessExpected})  & $\checkmark$ (Cor. \ref{betweennessconfidenceroot}) &  yes (Th. \ref{BetweennessPersistenceRoot}) \\ 
         &&&& \\
        \hline
        &&&& \\
        degree &  o(1) (\cite{eslavanumverticesdegree}) & $\Omega(n^{1-\log(2)})$ (\cite{eslavanumverticesdegree}) & $\times$ (\cite{eslavanumverticesdegree}) & no (\cite{eslavanumverticesdegree}) \\ 
        &&&& \\\hline
    \end{tabular}
    \label{mainresults}
\end{table}

\begin{table}[ht!]
    \centering
    \small
    \begin{tabular}{|c|c|c|c|}
    \hline
        &&& \\
           & $\EXP I_n$  & Tightness of $I_n$ & \begin{tabular}{@{}c@{}}Persistence \\ of center\end{tabular}    \\ 
          &&& \\
          \hline &&& \\
         \begin{tabular}
         {@{}c@{}}Jordan, \\ closeness \end{tabular}  & $<5/2+\mathcal{O} (n^{-1})$ (\cite{MoonJordan}) & yes (\cite{MoonJordan})  & yes (\cite{persistanceJogLoh})  \\ 
         &&& \\
         \hline
         &&& \\
         rumor  & $<5/2+\mathcal{O} (n^{-1})$ (\cite{MoonJordan})  & yes (\cite{MoonJordan}) & yes (\cite{persistanceJogLoh}) \\ 
         &&& \\
         \hline
         &&& \\
        betweenness  & $< 20$ (\cite{distributionBetweenness})  &  yes (\cite{distributionBetweenness})  &  yes (Th. \ref{BetweennessPersistenceCentroid})  \\
        &&& \\
        \hline
        &&& \\
         degree  & $\Omega \left(n^{1-{1 \over 2\log(2)}}\right)$ (\cite{BanerjeeBhamidiHubs}) & no (\cite{eslavadegree}) & no (\cite{eslavadegree})\\ 
         &&& \\
         \hline
    \end{tabular}
    \label{mainresults2}
\end{table}

In Section \ref{sec:prelim} we present some basic properties of the centrality measures. In Sections \ref{sec:jordan}--\ref{sec:closeness}, we prove the main results regarding the five centrality measures.

 \section{Preliminaries}
 \label{sec:prelim}
 
 In this section, we discuss some basic properties of the centrality measures. We denote the centrality measure by $\psi_T(v)$, where the specific centrality measure is clear from context.
 The first lemma gathers some properties of Jordan centrality. See Jordan \cite{jordan},
 Jog and Loh \cite[Lemma 13]{sublinearPA} and
 \cite[Lemma 2.1]{persistanceJogLoh}.
 
\begin{lemma} \label{propertiesJordan}
    Let $T_n$ be a tree on n vertices. Then Jordan centrality satisfies the following properties.
    \begin{enumerate}\renewcommand{\labelenumi}{(\roman{enumi})}
        \item for any two vertices $v,u\in T_n$ we have 
        \begin{equation*}
            \psi_T(u) \le \psi_T(v) \iff |(T_n,v)_{u\downarrow} | \geq |(T_n,u)_{v\downarrow} |~.
        \end{equation*}

        \item  $T_n$ can have at most two Jordan centers. If there are two, they must be neighbors. 
        \item If $v^*$ is a Jordan center, then $\psi_T(v^*) \leq \frac{n}{2}$.

    \end{enumerate}
  \end{lemma}
  
  Zelinka \cite{propertiesCloseness} observed that the Jordan and closeness centers coincide with the centroid. We show that the rumour center is also the centroid of the tree.
  Some of the following properties were established for the closeness centrality by Zelinka \cite{propertiesCloseness}.

\begin{lemma}\label{propertiesClosenessrumor}
    Let $T_n$ be a tree on $n$ vertices, and let $\psi_T(v)$ denote either closeness or rumor centrality. Then the following holds.
    \begin{enumerate}\renewcommand{\labelenumi}{(\roman{enumi})}
        \item For any two vertices $v,u \in T_n$ such that $\{u,v\} \in E(T_n)$, 
    \begin{equation*}
        \psi_T(u) \le \psi_T(v) \iff |(T_n,v)_{u\downarrow} | \geq |(T_n,u)_{v\downarrow} |~.
    \end{equation*}
    \item The vertex or pair of vertices most central according to the Jordan centrality, rumor centrality and closeness centrality coincide. 
    \item (Monotonicity) In any path rooted at a centroid $v_1$, of the form $(v_1,v_2,\ldots,v_{\ell})$, we have $\psi_T(v_i) < \psi_T(v_{i+1})$ for $2\leq i \leq \ell-1$ and $\psi_T(v_1) \leq \psi_T(v_{2})$ with equality only when $v_1,v_2$ are the two unique centroids. 
    \end{enumerate}
\end{lemma}
\begin{proof}
    (i) follows directly from the definition of the centrality measures. For the closeness centrality, note that  
    \begin{equation*}
        \mathcal{C}_{T_n}(v) = \sum_{u \in V(T_n)  \backslash v} \text{dist}(v,u) =  \sum_{u \in V(T_n)  \backslash v} |(T_n,v)_{u\downarrow}|~,
    \end{equation*}
    implying property (i). 
    (ii) By (i) and the definition of centroid, if $v$ is a closeness or rumor center, then $|(T_n,v)_{u\downarrow}| \le \frac{n}{2}$ for all $u \in N_n(v)$, and there can be at most two vertices that fulfill this condition, which by Lemma \ref{propertiesJordan}  are the Jordan centers. Thus, the Jordan, closeness and rumor centers coincide. 
    
    (iii) Consider a path $(v_1,v_2,\ldots,v_{\ell})$ away from a centroid. The first vertex must have, by the definition of centroid, $\mathcal{C}_{T_n}(v_1) \leq \mathcal{C}_{T_n}(v_{2})$, and by (ii), this is strict unless they are the pair of centroids. Consider now a vertex $v_{i+1}$ in the path. Assuming the induction hypothesis, $|(T_n,v_{i-1})_{v_i\downarrow}| < |(T_n,v_{i})_{v_{i-1}\downarrow}|$ and as $(T_n,v_{i})_{v_{i+1}\downarrow} \subset (T_n,v_{i-1})_{v_i\downarrow} $ as well as $(T_n,v_{i})_{v_{i-1}\downarrow} \subset (T_n,v_{i+1})_{v_{i}\downarrow}$, we have $\mathcal{C}_{T_n}(v_i)<\mathcal{C}_{T_n}(v_{i+1})$ for all $2\leq i \leq \ell-1$. 
\end{proof}

We recall a result by Jog and Loh \cite{persistanceJogLoh}.
\begin{lemma} {(Lemma 3.2 in \cite{persistanceJogLoh})} \label{twoverticesJordan}
For any two distinct vertices $u,v\in \cup_{n=1}^{\infty} V(T_n)$, let $\mathcal{S}_{\psi}=\{n : \psi_{T_n}(v) = \psi_{T_n} (u)\}$, for $\psi_{T_n}(v)$ the Jordan centrality of vertex $v$ in $T_n$. Then, $|\mathcal{S}_{\psi}| <\infty$ with probability 1. 
\end{lemma}

We now prove that, for both rumor and closeness centrality, the number of times at which two neighboring vertices have equal centrality is almost surely finite.

\begin{lemma}\label{finitechangeneighbors}
  For any two distinct neighbouring vertices $u,v\in \cup_{n=1}^{\infty} V(T_n)$, let $\mathcal{S}_{\mathcal{C}}=\{n : \mathcal{C}_{T_n}(v) = \mathcal{C}_{T_n} (u)\}$, where $\mathcal{C}_{T_n}(v)$ is the closeness centrality of vertex $v$ in $T_n$.
  Also, let $\mathcal{S}_{\varphi}=\{n : \varphi_{T_n}(v) = \varphi_{T_n} (u)\}$, where $\varphi_{T_n}(v)$ is the rumor centrality of vertex $v$ in $T_n$. Then, $|\mathcal{S}_{\mathcal{C}}| <\infty$ with probability 1, and $|\mathcal{S}_{\varphi}| <\infty$ with probability 1. 
\end{lemma}
\begin{proof}
Let $i$ and $j$ be neighboring vertices indexed by their arrival time. By Lemma  \ref{propertiesClosenessrumor}, for $n\ge i,j$,
$\psi_{T_n}(i) \le \psi_{T_n}(j) \iff |(T_n,i)_{j\downarrow} | \le |(T_n,j)_{i\downarrow} |$. The size of one of these subtrees increases by $1$ at step $n+1$, whereas the other remains the same size. This implies that if at time $n+1$,  $\psi_{T_{n+1}}(i) \ge \psi_{T_{n+1}}(j) \iff |(T_{n+1},i)_{j\downarrow} | \ge |(T_{n+1},j)_{i\downarrow} |$ then either $|(T_{n},i)_{j\downarrow} | = |(T_{n},j)_{i\downarrow} |$ or $|(T_{n+1},i)_{j\downarrow} | = |(T_{n+1},j)_{i\downarrow} |$. Thus, for every change in the ordering of the more central of neighboring vertices $i$ and $j$, at some time, the subtrees $(T,i)_{j\downarrow}$ and $(T,j)_{i\downarrow}$ must have been of equal size.

Without loss of generality, assume $i<j$ such that $(i,j)$ is an edge. The subtree sizes $|(T_n,j)_{i\downarrow} |, |(T_n,i)_{j\downarrow} |$ evolve over time $n\ge j$ as a standard (two-color) Pólya-Eggenberger urn with starting state $(j-1,1)$ at time $n=j$. By well-known almost surely convergence results (see, e.g., \cite{janson2004functional} \cite{mahmoud}), the proportion of the size of the subtree to the size of the vertex set converges to a limit 
\begin{equation*}
    \frac{|(T_n,j)_{i\downarrow}|}{n} \xrightarrow[n \to \infty]{a.s.} \beta \sim \text{Beta}(j-1,1).
\end{equation*}

By absolute continuity of the beta distribution, $\PROB(\text{Beta}(j-1,1)=1/2)=0$, so with probability one $|(T_n,j)_{i\downarrow}|/n$ equals $1/2$ finitely many times. This implies that almost surely, the subtree sizes are equal only finitely many times. 
\end{proof}
Jog and Loh \cite{persistanceJogLoh} proved that the centroid is persistent, a property that extends to the Jordan, rumor, and closeness centers:
\begin{theorem}{(Theorem 1 in \cite{persistanceJogLoh})} \label{JordanPersistence}
    In a \textsc{urrt}, with probability one, there exists a time $N$ and a node $ v^* \in T_N$ such that $ v^*$ is the unique centroid of $T_n$ for all $ n \ge N$. 
\end{theorem}


We often refer to the max-Dickman-Goncharov distribution, as described by Molchanov and Panov \cite{dickman}. 
\begin{definition} {(The max-Dickman-Goncharov distribution.)}
Let $U_1,U_2,U_3,\ldots$ be a sequence of independent uniform random variables on $[0,1]$.
  A random variable $D$, distributed as
\begin{equation*}
    D = \max \{U_1,(1-U_1)U_2,(1-U_1)(1-U_2)U_3,\ldots\},
\end{equation*}
has the max-Dickman-Goncharov distribution.
\end{definition}
 The probability density function $g(x)$ of $D$ is unimodal, has maximum value $2$ at $x=1/2$, remains bounded away from $0$ on its support,  and equals $1/x$ for $x\in[1/2,1]$. For $x\in [0,1]$, the density function can be found by solving 
\begin{equation*}
    xg(x)=\int_0^{x/(1-x)}g(u)du.
\end{equation*}

The max-Dickman-Goncharov distribution is useful as it is the limiting distribution of the size of the largest subtree rooted at a neighbor of the root, after division by $n$. 
\begin{lemma}\label{maxSubtreeDickman}
In a \textsc{urrt},
    \begin{equation*}
        \frac{\max_{v\in N_n(1)}|(T_n,1)_{v\downarrow}|}{n} \xrightarrow[n \to \infty]{a.s.} \mathcal{X}\sim D~,
    \end{equation*}
    
    where $D$ is the max-Dickman-Goncharov distribution. 
\end{lemma}
\begin{proof}
    It is well-known (see, e.g., Pitman \cite{pitman}) that for all $v_i \in N_n(1)$ the $i^{th}$ neighbor of the root,
    \begin{equation*}
        \frac{|(T_n,1)_{v_i\downarrow}|}{n} \xrightarrow{\text{a.s.}}\mathcal{X}_i \sim\prod_{k=1}^{i-1} (1-U_k)U_i \quad\text{ as }\; n \to \infty~,
    \end{equation*}
    
    for $U_1,U_2,\ldots,U_i$ independent uniform random variables on $[0,1]$. By the definition of the max-Dickman-Goncharov distribution and the continuous mapping theorem, the size of the largest subtree of a neighbor of the root, after division by $n$,
    converges almost surely to a limit with max-Dickman-Goncharov distribution.
\end{proof}

\section{Jordan centrality}
\label{sec:jordan}

In this section, we discuss properties of the Jordan centrality rank of the root and the index of the Jordan center.
Recall that the Jordan center coincides with the tree's centroid and with the closeness and rumor centers.
Moon \cite{MoonJordan} noted that the centroid is unique with high probability.
Haigh \cite{firstpapernetworkarcheology} showed that
$\PROB(R_n=1)=\PROB(I_n=1) \to 1 -\log(2) \text{ as }n\to\infty$. See also Moon \cite{MoonJordan}.
Moon \cite{MoonJordan} also found an explicit formula for the expected value of $I_n$.
In particular, we have $\EXP I_n \to 5/2$ as $n\to \infty$.
Since $\EXP I_n$ is bounded, this implies tightness of the sequence $\{I_n\}$ via Markov's inequality.
Jog and Loh \cite{persistanceJogLoh} showed that the centroid $I_n$ is persistent.
Our results concern tail bounds for the rank $R_n$ of the root (Theorem \ref{JordanUpperBoundProbability})
and the persistence of $R_n$ (Theorem \ref{persistencerankrootJordan}).





\subsection{Root estimation and expected rank of the root}

Bubeck, Devroye, and Lugosi \cite{findingadam} showed that Jordan centrality is good for
root estimation (i.e., $\{R_n\}$ is a tight sequence). In particular, they showed that
for all positive integers $x$,
$\PROB(R_n>x)\leq \frac{11\log(1/x)}{x}$. They also noted that $\PROB(R_n>x)\geq \frac{1}{3x}$.
The main result of this section improves the upper bound, obtaining optimal bounds up to a constant factor.

\begin{theorem} \label{JordanUpperBoundProbability}
There exists a constant $c>0$ such that for all $n\ge 1$ and for all $x\in[n],
$\begin{equation*}
        \PROB(R_n>x)\leq \frac{c}{x}~,
      \end{equation*}
\end{theorem}
\begin{proof}
Let $\mathcal{S}_n$ be the set of nodes in the tree whose Jordan centrality is at most that of the root, and let $D_n^*=\max_{v\in N_n(1)}|(T_n,1)_{v\downarrow}|$. Then,
if $D^*_n/n<1/2$,  we have $\mathcal{S}_n =\{1\}$. If $D^*_n/n \ge 1/2$, however, then 
$\mathcal{S}_n$ is a subtree rooted at $1$ that contains all nodes $v$ with
    \begin{equation*}
            1-\frac{|(T_n,1)_{v\downarrow}|}{n} < \frac{D_n^*}{n}~.
    \end{equation*}
By Lemma \ref{propertiesJordan}, if $\frac{D_n^*}{n}<\frac{1}{2}$, then the root is the centroid. If $\frac{D_n^*}{n}\ge\frac{1}{2}$, then the root is not the centroid, and by monotonicity the vertices in $\mathcal{S}_n$ form a subtree.


Let $D^{\dagger}_n=\frac{D_n^*}{n}$ and let $D^{\dagger}=\lim_{n\to\infty}D^{\dagger}_n$.  By Lemma \ref{maxSubtreeDickman}, $D^{\dagger}$ exists almost surely.
Then, for any $x\in(0,1)$ and any $t\in \mathbb{N}$,
\begin{align} \label{equationRank}
   \limsup_{n\to\infty}\PROB(R_n\ge t)\le \limsup_{n\to\infty}\PROB(\mathcal{S}_n\ge t,D^{\dagger}_n\le 1-x) + \limsup_{n\to\infty}\PROB(D^{\dagger}_n>1-x)~. 
\end{align}
We embed the \textsc{urrt} process in the Ulam-Harris tree defined in section \ref{notation}. Write $T_{\infty}$ for the Ulam-Harris tree. The \textsc{urrt} at time $n$, $T_n$, is a subtree of $T_{\infty}$, and $T_1\subset T_2\subset \cdots\subset T_{\infty}$. 

Define the value of a node $v$ by $\text{Value}(v)=
\lim_{n\to\infty}|(T_n,1)_{v\downarrow}|/n$. The evolution of $(|(T_n,1)_{v\downarrow}|,n-|(T_n,1)_{v\downarrow}|)$ is equivalent to that of a two colour Pólya-Eggenberger urn, and by standard results, the limit of $|(T_n,1)_{v\downarrow}|/n$ exists almost surely.

For a fixed vertex $v\in T_{\infty}$ with value $X_v = \text{Value}(v)$, the joint distribution of the values of its children $v_1,v_2,\ldots$, is 
\[
\left(\text{Value}(v_1), \text{Value}(v_2), \ldots \right) 
\sim \left( X_vU_1,X_v(1-U_1)U_2,\ldots \right)~,
\]

where $\{U_i\}_{i=1}^{\infty}$ are independent uniform $[0,1]$ random variables. The  Jordan centrality of vertex $v$ follows the limit law given by
$$
\lim_{n\to\infty}\frac{\psi_n(v)}{n}=\max\left( 1-X_v, X_{v_2}, X_{v_3},\ldots\right) ~.
$$
Let $\mathcal{S}$ be the subset of vertices of $T_{\infty}$ that are as central as the root in $T_{\infty}$, that is, vertices $v$ with $X_v\ge 1-D^\dagger$.
The set $\mathcal{S}$ forms a subtree of $T_{\infty}$. By Theorem \ref{persistencerankrootJordan} below, the set of vertices that ever become more central than the root is almost surely finite. Thus, after some random finite time, no new vertex ever becomes more central than the root. This implies that $|\mathcal{S}|<\infty$. For a fixed vertex $v$ and for sufficiently large $n$, $v$ is more central than the root in $T_n$ if and only if it is more central than the root in $T_{\infty}$. 

If  $D^{\dagger}\le 1-x$, the subtree $\mathcal{S}$ of $T_{\infty}$ includes all vertices with $\lim_{n\to\infty}|(T_n,1)_{v \downarrow}|/n\ge x$.
Let us divide the vertices in $\mathcal{S}$ into three sets: the set $\mathcal{L}$ of leaves of the subtree, 
the set $\mathcal{B}$ of vertices with more than two children in the subtree, and the set $\mathcal{O}$ of vertices with one single child in the subtree. It is easy to see that $|\mathcal{B}|<|\mathcal{L}|$. Define similarly, for the partition of the set  $\mathcal{S}_n$, the set $\mathcal{L}_n$ of leaves of the subtree, the set $\mathcal{B}_n$ of vertices with more than two children in the subtree, and the set $\mathcal{O}_n$ of vertices with one single child in the subtree.


Consider a vertex $v$ in the Ulam-Harris tree. Let $u_1,u_2,\ldots $ be the children of $v$, and let $u^{(k)}$ be the vertex among $\{u_i\}_{i=1}^{\infty}$ whose value is the $k$-th largest. (Since the value has an absolutely continuous distribution, ties occur with probability zero.)
If $u^{(1)}=u_1$, then
\begin{align}
    \text{Value}(u^{(2)}) &\overset{d}{=} X_v  \max\{(1-X_1)U_2,(1-X_1)(1-U_2)U_3, \ldots\}
    \nonumber\\&  \overset{d}{=} X_v (1-X_1) \max\{U_2,(1-U_2)U_3, \ldots\}
    \nonumber\\&  \overset{d}{=} X_v (1-X_1) D_2~, \label{secondlargest}
\end{align}
where $X_1= \text{Value}(u_1)/ \text{Value}(v)$  and $D_2$ is max-Dickman-Goncharov distributed, independent of $X_1$.

Let $\mathcal{O}'$ denote the set of siblings of nodes in $\mathcal{O}$ in the Ulam-Harris tree.
Fix $x\in (0,1)$. If $D^\dagger \le 1-x$, then for any $u\in \mathcal{O}$,
$\text{Value}(u) \ge x$, and
there exists $v\in\mathcal{O}'$,
a  sibling of $u$, such that
$X_v<x$. To see this, observe that $v$ is less central than the root.

Let $x\in(0,1/100]$ and $x' = x/X_v$. Then
\begin{align}
    \PROB&\left(\text{Value}(u^{(2)})\ge \frac{x}{2} \middle| u^{(1)}\in \mathcal{O},u^{(2)}\in \mathcal{O}', D^\dagger \le 1-x, u^{(1)}=u_1\right) \nonumber\\
     &= \PROB\left(\text{Value}(u^{(2)})\ge \frac{x}{2} \middle| \text{Value}(u^{(1)})\ge x,\text{Value}(u^{(2)}) < x, D^\dagger \le 1-x, u^{(1)}=u_1\right)
     \nonumber\\
     &= \PROB\left((1-D_1)D_2\ge \frac{x'}{2} \middle| D_1\ge x',(1-D_1)D_2 < x', D^\dagger \le 1-x, u^{(1)}=u_1\right)
    \nonumber\\&\ge \frac{1}{2}~. \label{lowerboundhalf}
\end{align}
where in the second equality we used (\ref{secondlargest}). Note that this bound holds for any conditioning  $u^{(1)}=u_k$ for fixed $k>0$. 

By the lower-bound (\ref{lowerboundhalf}), the sizes of the subtrees rooted at vertices in $\mathcal{O}'$ are lower-bounded by a random variable distributed as $\frac{x}{2}\text{Bin}(|\mathcal{O}|,1/2)$.
Hence, 
\begin{align*}
   \limsup_{n\to\infty}&\PROB(R_n\ge t,D^{\dagger}_n\le 1-x) = \limsup_{n\to\infty}\PROB( |\mathcal{L}_n| + |\mathcal{B}_n| + |\mathcal{O}_n| \ge t,D^{\dagger}_n\le 1-x) \\ &\le \PROB( 2|\mathcal{L}| +  |\mathcal{O}'| \ge t,D^{\dagger}\le 1-x)\\ &\le \PROB( |\mathcal{L}|  \ge t/3,D^{\dagger}\le 1-x) + \PROB( |\mathcal{O}'|  \ge t/3,D^{\dagger}\le 1-x) \\ &\le \mathbbm{1}_{[t \le 3/x]} + \PROB\left( \frac{x}{2}\text{Bin}(t/3,1/2)  \le1 \right) \\&= \mathbbm{1}_{[t \le 3/x]} + \PROB\left(\text{Bin}(t/3,1/2)-\frac{t}{6}  \le \frac{2}{x}-\frac{t}{6}\right) \\ & \le \mathbbm{1}_{[t \le 3/x]} + \mathbbm{1}_{[t \le 24/x]} +\mathbbm{1}_{[t > 24/x]}\PROB\left(\text{Bin}(t/3,1/2)-\frac{t}{6}  \le -\frac{t}{12}\right)  \\&\le \mathbbm{1}_{[t \le 3/x]} + \mathbbm{1}_{[t \le 24/x]} +\mathbbm{1}_{[t > 24/x]}\frac{t/12}{t/12+t^2/144} ~,
\end{align*}
where the last inequality follows from the Chebyshev-Cantelli inequality. The third inequality follows after observing that any vertex $v$ that is a leaf in $\mathcal{S}$---and thus belongs to $\mathcal{L}$---is more central than the root, and is in the largest subtree of the root, $D^*$. 
By Lemma \ref{propertiesJordan}, the value of $v$ satisfies $\text{Value}(v) \ge (1-D^\dagger)$. For two distinct vertices, $u,v\in \mathcal{L}$ the subtrees $(T,1)_{v\downarrow},(T,1)_{u\downarrow}$ are disjoint. As there are $|\mathcal{L}|$ such vertices, and the sum of values for vertices with disjoint descendants is at most one, we have $|\mathcal{L}|(1-D^{\dagger})\le 1$. Therefore, if $D^{\dagger}\le 1-x$, we require $|\mathcal{L}|x \le 1$.

Thus,
\begin{equation*}
    \limsup_{n\to\infty}\PROB( |\mathcal{L}_n| + |\mathcal{B}_n| + |\mathcal{O}_n| \ge t,D^{\dagger}_n\le 1-x) = \begin{cases}
        1 & \text{ if } t\le 24/x~ \\ \frac{12}{12+t}
    & \text{ if } t > 24/x~.
    \end{cases} 
\end{equation*}

Together with equation (\ref{equationRank}), 
\begin{align*}
\limsup_{n\to\infty}\PROB(R_n\ge t) &\le \limsup_{n\to\infty}\PROB(R_n\ge t,D^{\dagger}_n\le 1-x) + \limsup_{n\to\infty}\PROB(D^{\dagger}_n>1-x) \\&\le \limsup_{n\to\infty}\PROB( |\mathcal{L}_n| + |\mathcal{B}_n| + |\mathcal{O}_n| \ge t,D^{\dagger}_n\le 1-x) + 2x\\& \le  \mathbbm{1}_{[t \le 24/x]} +\mathbbm{1}_{[t > 24/x]}\frac{12}{12+t} +2x~.
\end{align*}In the second inequality, we used the fact that the limit of $D^{\dagger}_n$ is max-Dickman-Goncharov distributed, and the density of the max-Dickman-Goncharov distribution is bounded by $2$. Thus, for $t>24/x$,
\begin{equation*}
    \limsup_{n\to\infty} \PROB(R_n\ge t) \le  \frac{12}{12+t} +2x~.
\end{equation*}
By setting $x=25/t$, for all $t\ge 1$,
\begin{equation*}
    \limsup_{n\to\infty} \PROB(R_n\ge t) \le  \frac{12}{12+t} + \frac{50}{t}~.
\end{equation*}
\end{proof}

By integrating the bounds of Theorem \ref{JordanUpperBoundProbability} and noting that the lower bound follows
from the above-mentioned observation of \cite{findingadam}, we obtain

\begin{corollary} \label{JordanExpected}
    The Jordan rank of the root satisfies $\EXP R_n = \Theta(\log(n))$. 
  \end{corollary}

We also obtain the following bound for the size of the confidence set of the root, obtained by the $K$ most Jordan-central vertices.   

\begin{corollary}
    Let $H_{K}$ be the set of the top $K$ most central vertices according to the Jordan centrality. For every $\varepsilon>0$, if $K\ge 66/\varepsilon$,
then $\liminf_{n\to\infty} \PROB(1 \in H_{K})\ge 1-\varepsilon$. 
\end{corollary}

\subsection{Persistence}
To establish the persistence of the root's Jordan rank $R_n$, we first observe that the evolution of the subtrees adjacent to the root is stochastically equivalent to Hoppe’s urn model \cite{hoppeUrn}. This process begins with a single black ball; at each discrete time step, a ball is drawn uniformly at random. If the black ball is selected, it is returned to the urn alongside a ball of a new, unique color. If a colored (non-black) ball is drawn, it is returned with an additional ball of the same color. 

In the context of the tree, the black ball represents the root, and each color corresponds to a specific subtree rooted at a neighbor of the root. An attachment to the root (which creates a new subtree) is thus modeled by drawing the black ball. Our proof begins by showing that in Hoppe’s urn, a single color almost surely emerges as the unique, permanent leader in size.

\begin{theorem}\label{Hoppelargest}
    There exists a random variable $\tau < \infty$ such that for all $n\ge\tau$, the color with the largest number of balls in Hoppe's urn remains unchanged.
\end{theorem}
The proof relies on the following lemma, stated as in \cite{persistanceJogLoh}, first shown by Galashin \cite{Galashin2016}. 

\begin{lemma} \label{galashinLemma}
    Consider a $2$-dimensional random walk on the lattice $\mathbb{N}\times \mathbb{N}$. Denote by $W_n$ its position at time $n$, and assume
    \begin{align*}
        \PROB(W_{n+1}=(i+1,j)| W_n = (i,j) ) &\propto i~, \\ 
        \PROB(W_{n+1}=(i,j+1)| W_n = (i,j) ) &\propto j~,
    \end{align*}
(where  $f \propto g$ denotes that there exists a constant $c>0$ such that $f=cg$).
    For $M>2$, let $D_M$ be the event that the random walk starting at $W_0 = (M,1)$ reaches the diagonal at some future time. Then there exists a fixed polynomial $P$ such that 
    \begin{equation*}
        \PROB(D_M) \le \frac{P(M)}{2^M}.
    \end{equation*}
\end{lemma}
Recall that by the definition of Jordan centrality,
\begin{align*}
    \psi_n (v) \le \psi_n(1) \iff |(T_n,1)_{v\downarrow}| \ge |(T_n,v)_{1\downarrow}|~.
\end{align*}
The relative Jordan ranking of a vertex $n$ compared with the root evolves as a $2$-dimensional random walk on $\mathbb{N}\times \mathbb{N}$ with the transition probabilities as in Lemma \ref{galashinLemma} and with starting state $W_0 = (|(T_n,n)_{1\downarrow}|,1)$.

The next lemma follows from the almost sure convergence of the relative sizes of the subtrees to a Dirichlet-distributed random variable---see Jog and Loh \cite{persistanceJogLoh}. 
\begin{lemma} \label{lowerboundsubtreesize}
    For every $v\in \cup_{n=1}^{\infty}V_n$ there exists a random variable with absolutely continuous distribution $\xi_v$  such that 
    \begin{equation*}
        |(T_n,1)_{v\downarrow}| \ge \frac{n}{\xi_v}~,
    \end{equation*}
    for all $n\ge v$. 
\end{lemma}

\begin{proof}{(of Theorem \ref{Hoppelargest})}
    The proof follows the argument of Thörnbald \cite{dominatingcolour}.
Let $\mathcal{H}_n$ be the event that the color ``born'' at time $n$ ever becomes as large as the leading color at time $n$. Let $M_n$ be the size of the leading color at time $n$. We define the event $\mathcal{C}_m = \cap_{n\ge m} [M_n \ge n^{1/2}]$. 

Then, one may consider the events $[\mathcal{H}_n \cap \mathcal{C}_m]$, 
\begin{align*}
    \sum_{n=1}^{\infty} \PROB (\mathcal{H}_n \cap \mathcal{C}_m) &\le  \sum_{n=1}^{\infty} \frac{P(n^{1/2})}{2^{n^{1/2}}} < \infty,
\end{align*}
where $P$ is as in Lemma \ref{galashinLemma}. Hence, the events $[\mathcal{H}_n \cap \mathcal{C}_m]$ occur finitely often, almost surely. By Lemma \ref{lowerboundsubtreesize}, $\PROB(\mathcal{C}_m) \to 1$ as $m\to \infty$.  This implies that the events $\mathcal{H}_n$ occur finitely often almost surely. Hence, only finitely many colors can ever be the largest. 
In a Pólya urn of two colors, the size of the two colors are equal at most finitely many times almost surely (see proof of Lemma \ref{finitechangeneighbors}), hence among the finitely many colors that become the largest at some point in time, there is one that after some random but almost surely finite time is the only one that is the largest.
\end{proof}

We may rephrase  Theorem \ref{Hoppelargest} as follows:

\begin{corollary}\label{alternateformHoppeDominating}
  Let $\mathcal{I}_n=\arg \displaystyle \max_{v\in N_n(1)} |(T_n,1)_{v\downarrow}|$. There exists an integer-valued random variable
$\mathcal{I}>1$,  such that, almost surely,
    \begin{gather*}
        \lim_{n \to \infty} \mathcal{I}_n = \mathcal{I}~.
    \end{gather*}
\end{corollary}

To prove persistence, we need the following additional result. 
\begin{lemma}
    Let $D_n^*= \max_{v\in N_n(1) }|(T_n,1)_{v\downarrow}|$. There exists almost surely a finite time $M$ such that $n-|D_n^*| \ge \xi' n$ for all $n\ge M$, where $\xi'$  is a positive random variable with $\PROB(\xi'<\infty)=1$.
\end{lemma}
\begin{proof}
    This follows from Lemma \ref{lowerboundsubtreesize} as one can upper-bound the size of the subtree by a random variable $W<1$ times the tree size.
\end{proof}

\begin{theorem}\label{persistencerankrootJordan}
Let $R_n$ be the Jordan rank of the root at time $n$. Then, there exists an integer-valued random variable $R$ such that with probability one, $\lim_{n\to \infty} R_n = R$. 
\end{theorem}

\begin{proof}
    Define the events
    \begin{align*}
    E_n &= \left\{\cup_t \left[ |(T_t,1)_{n\downarrow}| \ge |(T_t,n)_{1\downarrow}| \right] \right\}~,  \\H_n &= \{n \in (T_n,1)_{\mathcal{I}\downarrow}  \}~, \\ M_m &= \left\{\cap_{n\ge m} \left[|(T_n,n)_{1\downarrow}| \ge n^{1/2} \right] \right\}~,
    \end{align*}
where $\mathcal{I}$ is defined in Corollary \ref{alternateformHoppeDominating}.
    By Theorem \ref{Hoppelargest}, only finitely many vertices not in the eventually largest subtree can become more central than the root almost surely. Hence, if one can prove that among the vertices that attach to the eventually largest subtree only finitely many become the most central, then this implies persistence (by combining the fact that for any two vertices they change order in the ranking only finitely many times almost surely, which follows from Lemma \ref{twoverticesJordan}). If $P$ is as in Lemma \ref{galashinLemma}, then 
    \begin{align*}
        \sum_{n=1}^{\infty} \PROB(E_n\cap H_n\cap M_m) \le \sum_{n=1}^{\infty}\frac{P(n^{1/2})}{2^{n^{1/2}}} < \infty~. 
    \end{align*}
    By the Borel-Cantelli lemma, only finitely many of the events $[E_n\cap H_n\cap M_m]$ occur over $n$ almost surely. Note that this holds for all $m$ and given that $\PROB(M_m)\to 1$ as $m \to \infty$, it follows that the number of vertices in the eventually largest subtree that become more central than the root is finite almost surely. This implies that the number of vertices that become more central than the root is finite almost surely. 
\end{proof}

\section{Betweenness centrality}
\label{sec:betweenness}

In this section, we discuss properties of betweenness centrality in uniform attachment trees.

Durant and Wagner \cite{distributionBetweenness} proved the tightness of the index of the betweenness center.
In particular, they showed that, for every $k\in [n]$,
\[
        \PROB(I_n \ge k) < 16\left(\frac{k}{3}+1\right)\left(\frac{3}{4}\right)^k~.
\]
In particular, this bound implies that $\EXP I_n < 20$.

To the best of our knowledge, betweenness centrality has not been studied in the context of root estimation or persistence.

\subsection{Probability that the root is the centroid}

We start by providing a lower bound on the probability that the betweenness center equals the root.
We make use of the representation $\mathcal{B}'_{T_n}(v) = \sum_{u\in T_n:uv\in E} |(T_n,v)_{u\downarrow}|^2$ for the betweenness centrality of a vertex. Note that smaller values of $\mathcal{B}'_{T_n}(v)$ correspond to more central vertices.

\begin{theorem}\label{probabilityBetweenness}
    Let $R_n$ be the betweenness rank of the root at time $n$. Then, 
    \begin{equation*}
        \liminf_{n\to\infty}\PROB(R_n=1) \ge  3-\sqrt{5}-\log(2) \approx 0.070784~.
    \end{equation*}
\end{theorem}
\begin{proof}
    Suppose $D_n^\dagger\defeq \frac{1}{n}\max_{v\in N_n(1)}|(T_n,1)_{v\downarrow} |<\varepsilon$ for some $\varepsilon\in(0,1]$, thus implying
    \begin{equation*}
        \mathcal{B}'_{T_n}(1) = \sum_{u \in N_n(1) }|(T_n,1)_{u\downarrow}|^2 \le n D_n^{\dagger}\sum_{u \in N_n(1) }|(T_n,1)_{u\downarrow}| \le n^2\varepsilon~.
    \end{equation*}
    On the other hand, for any vertex $v\in V(T_n)\backslash1$,
    \begin{equation*}
        \mathcal{B}'_{T_n}(v) = \sum_{u \in N_n(v) }|(T_n,v)_{u\downarrow}|^2 \geq (n-|(T_n,1)_{v\downarrow}|)^2 \ge (n-n\varepsilon)^2 = n^2(1-\varepsilon)^2~.
    \end{equation*}
    Therefore, if $n^2(1-\varepsilon)^2 \geq n^2\varepsilon $  then $\mathcal{B}_{T_n}(v) \geq \mathcal{B}_{T_n}(1)$.
    In particular,
 $\varepsilon\le \frac{3- \sqrt{5}}{2}$ implies that  $\mathcal{B}'_{T_n}(v) \geq \mathcal{B}'_{T_n}(1)$.
 Hence, if $D_n^\dagger< \frac{3- \sqrt{5}}{2}$, then all vertices are less central than the root.
 Using the fact that $D^\dagger_n$ converges in distribution to a max-Dickman-Goncharov random variable, we obtain
\[
        \lim_{n\to \infty}\PROB\left(D_n^\dagger\leq\frac{3- \sqrt{5}}{2} \right) \ge 3-\sqrt5-\log(2)~.
\]
\end{proof}

\subsection{Tightness, expected rank of the root and root estimation}

In this section, we investigate the quality of ranking vertices by their betweenness centrality
for root estimation. We show that betweenness centrality has similar properties to those
of Jordan centrality in this regard. The key is the following result whose proof uses
similar arguments as those of the proof of Theorem \ref{JordanUpperBoundProbability}.

\begin{lemma} \label{BetweennessRootEstimation}
    Let $R_n$ be the betweenness rank of the root at time $n$. Then, there exists a constant $c>0$ such that for all $x\in [n]$,
    \begin{equation*}
        \PROB(R_n>x)\leq \frac{c}{x}~.
    \end{equation*}
\end{lemma}

\begin{proof}
    The proof of Lemma \ref{BetweennessRootEstimation} is analogous to that of Theorem \ref{JordanUpperBoundProbability} and is omitted for brevity. The main difference is that, keeping the notation as in the proof of Theorem \ref{JordanUpperBoundProbability}, the set $S_n$ is now defined by $S_n=\{v\in T_n: (n-|(T_n,1)_{v\downarrow}|)^2<n\max_{v\in N_n(1)}|(T_n,1)_{v\downarrow}|\}\cup\{1\}$, and similarly $S=\{v\in T_{\infty}:(1-\text{Value}(v))^2<D^\dagger\}\cup\{1\}$. With these modifications, the proof proceeds along the same lines.
\end{proof}

To see that the upper bound is of optimal order, observe that leaves are the least central
among all vertices. Hence, the argument of \cite{findingadam} mentioned in the discussion on Jordan centrality
applies. In particular, with probability $1/(n-1)$, the root is a leaf in $T_n$.

\begin{theorem}\label{BetweennessExpected}
    Let $R_n$ be the betweenness rank of the root at time $n$. Then, $\EXP R_n = \Theta (\log (n))$. 
  \end{theorem}

\begin{corollary} \label{betweennessconfidenceroot}
     Let $H_{K}$ be the set of the top $K$ most central vertices according to the betweenness centrality. Then, $\exists \;C>0$ such that if $K = {C}/{\varepsilon}$ one has $\liminf_{n\to\infty} \PROB(1 \in H_{K})\ge 1-\varepsilon$. 
\end{corollary}

\subsection{Persistence}

In this section, we prove the persistence of the betweenness center and also of the rank of the root.
We first show the persistence of the relative ranking among two vertices.
Our main result requires the following lemmas. 

\begin{lemma}\label{betweenConvergence1}
    Let $i,j\in V(T_n)$ be two vertices with $i<j$. Then there exist random variables $X,Y,Z$ such that
    $$
        \frac{|(T_n,i)_{j\downarrow}|}{n} \xrightarrow[\substack{n \to \infty}]{\text{a.s.}} X~,
        $$
        $$
        \frac{|(T_n,j)_{i\downarrow}|}{n} \xrightarrow[\substack{n \to \infty}]{\text{a.s.}} Y~,
        $$
        and
        $$
        \frac{n-|(T_n,i)_{j\downarrow}|-|(T_n,j)_{i\downarrow}|}{n} \xrightarrow[\substack{n \to \infty}]{\text{a.s.}} Z~.
    $$
    Here $Y$ had $\text{Beta}(1,j-1)$ distribution.
    If $i$ is on the path from the root to $j$, then $X\sim \text{Beta}(K-1,1)$, where $K$ is a random variable with $i<K\le j$. Otherwise, $X\sim \text{Beta}(1,i-1)$.
    
  \end{lemma}
  
\begin{proof}
Suppose $i$ is not in the path from the root to $j$. Then, $(T_n,i)_{j\downarrow} = (T_n,1)_{j\downarrow}$ and similarly  $(T_n,j)_{i\downarrow} = (T_n,1)_{i\downarrow}$. By standard results for Pólya urn models, the convergence follows. 

If $i$ is in the path from the root to $j$, it still holds that $(T_n,i)_{j\downarrow} = (T_n,1)_{j\downarrow}$, so convergence follows for this subtree by standard arguments. In contrast, $(T_n,j)_{i\downarrow} = T_n \backslash(T_n,1)_{K\downarrow}$, where $K$ is the first vertex in the path from $i$ to $j$. Therefore, $|(T_n,j)_{i\downarrow}| = n- |(T_n,1)_{K\downarrow}|$. This implies that this can be modelled as a Pólya urn with $K-1$ balls of one color and one ball of the other color. The statement follows from standard convergence results of Pólya urns. 

The last convergence follows directly from the above.
\end{proof} 

Lemma \ref{betweenConvergence1} implies the following:

\begin{lemma}\label{betweenConvergence2}
  Let $i,j$ be two distinct vertices. Then, recalling that
  $\sigma_{i,j}$ denotes the path between vertices $i,j$,
    \begin{equation*}
       \sum_{\substack{v,k\in N_n(i) \\v,k \notin \sigma_{i,j} \\  v\ne k}} \frac{|(T_n,j)_{v\downarrow}||(T_n,j)_{k\downarrow}|}{n^2} \xrightarrow[\substack{n \to \infty}]{\text{a.s.}} X_i~,
    \end{equation*}
    where  $X_i$ is a random variable with a continuous distribution. 
\end{lemma}

\begin{lemma}\label{betweenPersistanceTwoVertices}
    For any two distinct vertices, there exists a random time $N$ with $\PROB(N<\infty)=1$  such that for all $n\ge N$ the relative ordering of their betweenness centrality does not change. 
\end{lemma}
\begin{proof}
For any two vertices $i,j\le n$, let $A_n=n-|(T_n,j)_{i\downarrow}|-|(T_n,i)_{j\downarrow}|$. Then $\mathcal{B}_{T_n}(i) \le \mathcal{B}_{T_n}(j)$ if and only if
\begin{align*}
    &|(T_n,j)_{i\downarrow}|A_n + \sum_{\substack{v,k\in N_n(i) \\v,k \notin \sigma_{i,j} \\  v\ne k}}|(T_n,j)_{v\downarrow}||(T_n,j)_{k\downarrow}| \\
    &\le |(T_n,i)_{j\downarrow}|A_n + \sum_{\substack{v,k\in N_n(j) \\v,k \notin \sigma_{i,j} \\  v\ne k}}|(T_n,i)_{v\downarrow}||(T_n,i)_{k\downarrow}|~.
\end{align*}
Lemma \ref{betweenConvergence1} and Lemma \ref{betweenConvergence2} imply the almost sure convergence of each term in the inequality as $n \to\infty$. Therefore, there exists an almost surely finite random time after which the relative ordering of the two does not change.
\end{proof}

The key technical ingredient in the proof of persistence of the rank of the root and the center is the following extension of Lemma \ref{galashinLemma}, which may be of independent interest. 

\begin{lemma}\label{diagonalRW}
    Let $(X_{a,t},Y_{a,t})$ be a two-color Pólya-Eggenberger urn evolving over time $t\ge0$, with $(X_{a,0},Y_{a,0})=(a,1)$, where $a\ge1$. Let $x\in(0,1)$ be fixed, and define the event $E_{a,x}= \left\{\frac{X_{a,t}}{a+1+t}<x \text{ for at least one }t \right\}$. Then
    $$\sum_{a=1}^{\infty} \PROB(E_{a,x})<\infty ~.$$
\end{lemma}
\begin{proof}
    Recall that by standard convergence properties of Pólya urns (see, e.g., \cite{janson2004functional} \cite{mahmoud}),
    $$\frac{X_{a,t}}{a+1+t}\xrightarrow[t\to\infty]{a.s.} \text B_a~,$$ 
    where $B_a$ is a Beta$(a,1)$ random variable. Then
    \begin{eqnarray*}
        \PROB(E_{a,x}) &\le & \PROB \left(B_a\le \frac{1+x}{2} \right)+ \PROB \left(E_{a,x},B_a >\frac{1+x}{2}\right)
        \\ &\le & \PROB \left(B_a\le \frac{1+x}{2} \right) + \sum_{t=1}^{\infty} \PROB \left(X_{a,t}= \lfloor  x(a+t+1)\rfloor,B_a >\frac{1+x}{2}\right) ~.
        \end{eqnarray*}
    Note that $\PROB(X_{a,t}= \lfloor  x(a+t+1)\rfloor)=0$ if $a>\lfloor  x(a+t+1)\rfloor$. Thus,
    \begin{align*}
        \PROB \left(X_{a,t}= \lfloor  x(a+t+1)\rfloor, B_a >\frac{1+x}{2}\right) \le \PROB \left(B_a >\frac{1+x}{2} \middle| X_{a,t}= \lfloor  x(a+t+1)\rfloor \right)
        \\ = \PROB \left( \text{Beta}(\lfloor  x(a+t+1)\rfloor,a+t+1-\lfloor  x(a+t+1)\rfloor) > \frac{1+x}{2}\right)~.
    \end{align*}
    We have used the fact that, conditionally on $ X_{a,t}= \lfloor  x(a+t+1)\rfloor$, the limiting proportion of $X$ balls is again beta distributed with parameters given by the composition at time $t$, which are $\lfloor  x(a+t+1)\rfloor$ $X$ balls and $a+t+1-\lfloor  x(a+t+1)\rfloor$ $Y$ balls. Setting $\alpha =\lfloor  x(a+t+1)\rfloor$ and $\beta =a+t+1-\lfloor  x(a+t+1)\rfloor$, we let $G_{\alpha}$ and $G_{\beta}$ be two independent gamma random variables with parameters $(\alpha,1)$ and $(\beta,1)$. Then,
    \begin{eqnarray*}
        \PROB \left(\text{Beta}(\alpha,\beta)> \frac{1+x}{2}\right) &\le & \PROB\left(\frac{G_{\alpha}}{G_{\alpha}+G_{\beta}} > \frac{1+x}{2} \right) 
        \\&= & \PROB(G_{\alpha}(1-x)-G_{\beta}(1+x)>0) 
        \\ &\le & \inf_{\lambda>0} \EXP \left[e^{\lambda(G_{\alpha}(1-x)-G_{\beta}(1+x))}\right]
        \\ &= & \inf_{\lambda>0} \left( \frac{1}{1-\lambda(1-x)}\right)^\alpha\left( \frac{1}{1+\lambda(1+x)}\right)^\beta
        \\ &\le & \left\{\inf_{0<\lambda<\frac{1}{1-x}} \left( \frac{1}{1-\lambda(1-x)}\right)^x\left( \frac{1}{1+\lambda(1+x)}\right)^{1-x}\right\}^{a+t+1}
        \\&= & \left(\frac{1}{2}\left(1+\frac{1}{x}\right)^x \right)^{a+t+1} =:\varphi(x)^{a+t+1} ~,
    \end{eqnarray*}
    where the last equality follows by computing the infimum, which is achieved at $\lambda= \frac{1}{1+x}$. Since $\varphi(1)=1$ and $\varphi'(x)>0$ for $x\in(0,1)$, we have $\varphi(x)<1$ for $x\in(0,1)$.  Therefore, 
    \begin{align*}
        \sum_{a=1}^{\infty}\PROB(E_{a,x}) &\le \sum_{a=1}^{\infty} \PROB \left(B_a\le\frac{1+x}{2} \right) + \sum_{a=1}^{\infty}\sum_{t\ge\frac{1-x}{x}a-1} (\varphi(x))^{a+t+1}
        \\ & \le \sum_{a=1}^{\infty} \left(\frac{1+x}{2}\right)^a + \sum_{a=1}^{\infty} \frac{x}{1-x}(t+1)(\varphi(x))^{t+1}
        \\ &< \infty ~.
    \end{align*}
\end{proof}

\begin{theorem}\label{BetweennessPersistenceCentroid}
 The index  $I_n$ of the betweenness center is persistent.
\end{theorem}

\begin{proof}
As observed by Durant and Wagner \cite{distributionBetweenness}, for a vertex $v$ to be the betweenness center, it is necessary that $|(T_n,1)_{v\downarrow}| \ge n/4$. Hence, by Lemma \ref{diagonalRW} with $x = 3/4$, almost surely only finitely many vertices have a subtree of size at least $n/4$ at some time $n$, which, together with Theorem \ref{betweenPersistanceTwoVertices}, proves persistence of the center.
\end{proof}

\begin{theorem}\label{BetweennessPersistenceRoot}
 The betweenness rank $R_n$ of the root is persistent.
\end{theorem}
\begin{proof}
For vertex $v$, define $A_{n,v}=(T_n,v)_{1\downarrow}$, $B_{n,v}=(T_n,1)_{v\downarrow}$, 
and let $N_n (v)$ denote the set of neighbors of $v$.
For $j\in N_n(v)$ not in the path from $1$ to $v$, let $M_{n,j} = |(T_n,v)_{j\downarrow}|$, and similarly for $i\in N_n(1)$ not in the path from $1$ to $v$, let $N_{n,i} = |(T_n,1)_{i\downarrow}|$. 
Let $\mathcal{J}_{n,v}=N_n(v)\backslash p(v)$, where $p(v)$ is the parent of vertex $v$, and $\mathcal{I}_{n,v}=N_n(1)\backslash u_1^v$, where $u_1^v$ denotes the first vertex in the path from the root to vertex $v$. 
Then, vertex $v$ is more central than the root if
\begin{align}
    \sum_{i\in\mathcal{I}_{n,v}} N_{n,j}^2 + (n-A_{n,v})^2 \ge \sum_{j\in\mathcal{J}_{n,v}} M_{n,j}^2 + (n-B_{n,v})^2 .
\end{align}
The latter event implies that
\begin{align}
(2n-A_{n,v}-B_{n,v})(B_{n,v}-A_{n,v}) \ge \sum_{j\in \mathcal{J}_{n,v}} M_{n,j}^2 - \sum_{i\in \mathcal{I}_n} N_{n,j}^2~. \label{ineqBetweenPersistence}
\end{align}

Let $\mathcal{S}$ be the set of vertices $v$ that at some time $n$ satisfy the inequality $B_{n,v}\ge A_{n,v}$:
$$\mathcal{S}= \{v\in\cup_{j=1}^{\infty}V_j: \cup_n [B_{n,v}\ge A_{n,v}]\}.$$
Then,
\begin{align}
    \PROB &(\cup_n [\mathcal{B}_n(v) \le \mathcal{B}_n(1)]) \nonumber\\
    & = \PROB (\cup_n [ \mathcal{B}_n(v) \le \mathcal{B}_n(1)]|v \in {\mathcal S})\PROB(v \in {\mathcal S}) \nonumber\\
    & \qquad + \PROB \left(\cup_n [ \mathcal{B}_n(v) \le \mathcal{B}_n(1)]|v \not\in {\mathcal S}\right)\PROB\left(v \not\in {\mathcal S}\right)
   \nonumber\\
   & \le  \PROB(v \in {\mathcal S}) + \PROB \left(\cup_n [\mathcal{B}_n(v) \le \mathcal{B}_n(1)]|v \not\in {\mathcal S}\right)~. \label{probabilityEv}
\end{align}
On the event $v \not\in {\mathcal S}$, inequality (\ref{ineqBetweenPersistence}) yields 
\begin{align}
    \left[ \sum_{j\in \mathcal{I}_n}N_{n,j}^2 \ge (A_{n,v}-B_{n,v})n \right] 
\end{align}
and this in turn implies the event
\begin{align}
    \left[ B_{n,v}n \ge \sum_{j\in \mathcal{I}_{n,v}}nN_{n,j}-\sum_{j\in \mathcal{I}_{n,v}}N_{n,j}^2 \right]
\end{align}
which, in turn, implies
\begin{align}
    \left[ B_{n,v} \ge \sum_{j\in \mathcal{I}_{n,v}} N_{n,j}\left( 1-\frac{N_{n,j}}{n}\right) \ge A_{n,v}\left( 1-\frac{\max_{j\in \mathcal{I}_{n,v}} N_{n,j}}{n}\right)
    \right]~. \label{inequalitysubtreeBetweenPersistent} 
\end{align}
Hence, the second term in (\ref{probabilityEv}) can be bounded as
\begin{align}
      \PROB &\left(\cup_n [\mathcal{B}_n(v) 
      \le \mathcal{B}_n(1)]|v \not\in {\mathcal S}\right) \nonumber \\
      &  \le  \PROB \left(\bigcup_n \left[
      B_{n,v} \ge A_{n,v}\left( 1-\frac{\max_{j\in \mathcal{I}_{n,v}} N_{n,j}}{n}\right)
      \right]
      \Biggm| 
      v \not\in {\mathcal S} \right) \nonumber \\  
      & \le \PROB \left(
      \bigcup_n
      \left[ 
      B_{n,v} \ge A_{n,v}\left( 1-\frac{\max_{j\in N_n(1)} N_{n,j}}{n}\right) 
      \right] 
      \Biggm| 
      v \not\in {\mathcal S} \right)~. \label{probabilityBoundAnBn}
\end{align}
By Lemma \ref{maxSubtreeDickman}, the term $\left( 1-\frac{\max_{j\in N_n(1)} N_{n,j}}{n}\right)$ converges almost surely to a limit. Thus, the probability in (\ref{probabilityBoundAnBn}), for large enough $n,v$ can be cast into the same setting as that of Lemma \ref{diagonalRW}, implying that the sum of the probabilities in (\ref{probabilityBoundAnBn}) among all vertices $v\in\mathbb{N}$ is finite. Therefore, 
\begin{align*} 
\sum_{v=2}^{\infty} &\PROB (\cup_n [\mathcal{B}_n(v) \le \mathcal{B}_n(1) ]   ) \\
&\le \sum_{v=2}^{\infty} \left(\PROB(v \in {\mathcal S}) + \PROB \left( \bigcup_n \left[ B_{n,v} \ge A_{n,v}\left( 1-\frac{\max_{j\in N_n(1)} N_{n,j}}{n}\right) 
\right]
\Biggm| 
v \not\in {\mathcal S} \right) \right) \\
&< \infty ~,
\end{align*}
where the convergence of the sum $\sum_{v=2}^{\infty} \PROB(v \in {\mathcal S})$ follows from Theorem \ref{persistencerankrootJordan}. Therefore, almost surely, only finitely many vertices become more central than the root. Together with Lemma \ref{betweenPersistanceTwoVertices}, this implies persistence of the rank of the root.
\end{proof}

\section{Degree centrality}
\label{sec:degree}

Degree centrality in a random recursive tree has been extensively analyzed by Eslava \cite{eslavadegree}, Eslava, López and Ortiz \cite{eslavanumverticesdegree}, and Banerjee and Bhamidi \cite{BanerjeeBhamidiHubs}. The results in this section follow from their results.

In particular, Eslava et al. \cite{eslavanumverticesdegree} showed that, for every $\varepsilon >0$, the degree-rank $R_n$ of the root is greater than $n^{1-\log(2)-\varepsilon}$ with probability tending to $1$.
In particular, $\PROB\{R_n=1\}=o(1)$ and
   \begin{equation*}
        \lim_{n \to \infty} \frac{\EXP R_n}{n^{1-\log(2)-\varepsilon}}  = \infty.
    \end{equation*} 
This implies that degree centrality is not suitable for root estimation and that $R_n$ is not persistent. 

Banerjee and Bhamidi \cite{BanerjeeBhamidiHubs} studied the index $I_n$ of the maximal degree vertex. They
showed that
    \begin{equation*}
        \frac{\log I_n}{\log n} \xrightarrow[n\to\infty]{P} 1-\frac{1}{2\log(2)}~.
    \end{equation*}
It follows that, for any $\varepsilon>0$, 
    \begin{equation*}
        \lim_{n \to \infty} \frac{\EXP I_n}{n^{1-1/(2\log(2))-\varepsilon}} = \infty.
    \end{equation*} 

    Eslava \cite{eslavadegree} proved that the distance of the degree center from the root satisfies a central limit theorem with mean $\mu \log(n)$, where $\mu=1-1/\log 4$, and variance $\sigma^2 \log(n)$ with $\sigma=1-1/\log 16$.
    This implies that $I_n$ is not tight, and it is easy to see that it is not persistent either.

\section{Rumor centrality}
\label{sec:rumor}

Shah and Zaman \cite{rumor} introduced the concept of rumor centrality in trees. They showed that
in a diffusion process over infinite regular trees, the likelihood of a vertex being the root
is determined by its rumour centrality. Crane and Xu \cite{crane2023} proved that the same holds
in uniform attachment trees as well. As a consequence, for every $\varepsilon>0$, the set of vertices of smallest
size that contains the root with probability at least $1-\varepsilon$ contains the $K(\varepsilon)$ most central 
vertices according to rumor centrality, for some positive integer $K(\epsilon)$.
Addario-Berry et al. \cite{addarioRumour} proved that one can take
$K(\varepsilon) \le C\exp \left(c\sqrt{\log(1/\varepsilon)}\right)$ for some positive constants $c,C$.
As shown by Bubeck et al. \cite{findingadam}, this bound is optimal up to the constant factors.
The upper bound of \cite{addarioRumour} implies that for all $n$ and $t\ge 1$,
\[
   \PROB\{ R_n \ge t\} \le \exp
   \left(-\left(\frac{1}{c}\log \left(\frac{t}{C}\right)\right)^2 \right)~.
 \]
 Thus, we also have that the expected rank of the root is $\EXP R_n = O(1)$. Note that none of the
 other centrality measures studied in this paper have this property.

 Recall from Lemma \ref{propertiesClosenessrumor} that the rumor center coincides with the centroid.
 Hence, $\PROB\{R_n=1\}=\PROB\{I_n=1\} \to 1- \log 2$ as $n\to \infty$ and $I_n$ is persistent.

 In the rest of this section, we show that the rank of the root $R_n$ is persistent.

 \subsection{Persistence of the rank of the root}
 
Let $V^{(k)} = \{v\in \N:\text{dist}_v(1,v) = k \}$ be the set of vertices at distance $k$ to the root. 

\begin{lemma}\label{persistenceDistance}
    Let $k>0$ be a constant. For each $v\in \N$ let $\mathcal{C}_v= \mathds{1}_{\{\exists \;n: \varphi_n(v) \le \varphi_n(1)\}}$. Then, almost surely, for any $k>0$, $ {\sum}_{v\in V^{(k)}} \mathcal{C}_v < \infty$. 
\end{lemma}
\begin{proof}
    The proof proceeds by induction on the distance from a vertex to the root. The base case is straightforward. Consider vertices $v\in V^{(1)}$, $v\le n$, that is, vertices that attach to the root. Then
    \begin{equation*}
        \varphi_n(v) \le \varphi_n(1) \iff |(T_n,1)_{v\downarrow}| \ge |(T_n,v)_{1\downarrow}|~.
    \end{equation*}
    By Lemma \ref{lowerboundsubtreesize} and Lemma \ref{diagonalRW} with $x=1/2$, it follows that for $E_v = \mathds{1}_{\{\exists n: |(T_n,1)_{v\downarrow}| \ge |(T_n,v)_{1\downarrow}|\}}$, almost surely $\sum_{v=1}^{\infty} E_v< \infty$, and therefore the base case holds. 

    Now assume that this holds for some $k>1$. Let $v\in V^{(k+1)}$, $v\le n$. 
    In the path from the root to vertex $v$, write $u_i$ for the $i$\textsuperscript{th} vertex in the path, $u_0$ for the root and $u_{k+1}=v$,  and write $T_{n,u_i}^v$ for the subtree rooted at vertex $u_i$ obtained by removing all edges in the path from the root to vertex $v$ in $T_n$. 
    Recall that for a vertex $v$, at distance $k+1$ to the root, to be more central than the root, we must have
    \begin{align*}
        \varphi_n(v) \le \varphi_n(1) \iff \prod_{i=1}^{k+1} |(T_n,1)_{u_i}| \ge \prod_{i=0}^{k} |(T_n,v)_{u_i}|~.
    \end{align*}
    Note the equivalence
    \begin{align}
        \frac{\varphi_n(1)}{\varphi_n(v)} \ge 1  
         \iff |(T_n,1)_{v\downarrow}| \ge |(T_n,v)_{1\downarrow}|\prod_{i=1}^{k} \frac{\sum_{\ell=0}^{k+1-i}|T_{n,u_{\ell}}^v|}{\sum_{j=i}^{k+1}|T_{n,u_j}^v|}~. \label{vlowerbound}
    \end{align}
    For any vertex $w \le n$ that is a sibling of $v$ in $T_n$, we have
    \begin{align}
        \frac{\varphi_n(1)}{\varphi_n(w)} 
        \ge 1   
        \iff |(T_n,1)_{w\downarrow}| \ge |(T_n,w)_{1\downarrow}|\prod_{i=1}^{k} \frac{\sum_{\ell=0}^{k+1-i}|T_{n,u_{\ell}}^v|}{\sum_{j=i}^{k+1}|T_{n,u_j}^v|}~.\label{siblinglowerbound}
    \end{align}
    For sibling vertices $v$ and $w$,  we have $(T_n,w)_{1\downarrow} = (T_n,v)_{1\downarrow}$. Thus, the right-hand sides of equations (\ref{vlowerbound}) and (\ref{siblinglowerbound}) are equal.  
\begin{figure}[!htb]
    \centering
\begin{tikzpicture}[
    vertex/.style={circle, draw=black, fill=black, thick, minimum size=5mm},
    dots/.style={},
    ellipsoid/.style={ellipse, draw=black, fill=black!20, thick, minimum width=7.5mm, minimum height=15mm},
    edge/.style={thick}
]

\node[vertex] (v1) at (0,0) {};
\node[vertex] (v2) at (1.3,0) {};
\node[vertex] (v3) at (2.6,0) {};
\node[dots] (v4) at (3.9,0) {$\ldots$};
\node[vertex] (v6) at (5.2,0) {};
\node[vertex] (v7) at (6.5,0) {};

\draw[edge] (v1) -- (v2);
\draw[edge] (v2) -- (v3);
\draw[edge] (v3) -- (v4);

\node[ellipsoid, anchor=north] at (v1.north) {};
\node[ellipsoid, anchor=north] at (v2.north) {};
\node[ellipsoid, anchor=north] at (v3.north) {};
\node[ellipsoid, anchor=north] at (v6.north) {};
\node[ellipsoid, anchor=north] at (v7.north) {};

\node at (0,0.6) {$u_0$};
\node at (1.3,0.6) {$u_1$};
\node at (2.6,0.6) {$u_2$};
\node at (5.2,0.6) {$u_k$};
\node at (6.5,0.6) {$v$};

\node at (0,-0.6) {$\scriptstyle T_{n,u_0}^v$};
\node at (1.3,-0.6) {$\scriptstyle T_{n,u_1}^v$};
\node at (2.6,-0.6) {$\scriptstyle T_{n,u_2}^v$};
\node at (5.2,-0.6) {$\scriptstyle T_{n,u_k}^v$};
\node at (6.5,-0.6) {$\scriptstyle T_{n,v}^v$};

\node[vertex] (v11) at (0,0) {};
\node[vertex] (v12) at (1.3,0) {};
\node[vertex] (v13) at (2.6,0) {};
\node[dots] (v14) at (3.9,0) {$\ldots$};
\node[vertex] (v16) at (5.2,0) {};
\node[vertex] (v17) at (6.5,0) {};

\draw[edge] (v11) -- (v12);
\draw[edge] (v12) -- (v13);
\draw[edge] (v13) -- (v14);
\draw[edge] (v14) -- (v16);
\draw[edge] (v17) -- (v16);

\end{tikzpicture}
\caption{Representation of the subtrees in the path from the root to vertex $v$.}
    \label{subtreesPathFigure}
\end{figure}
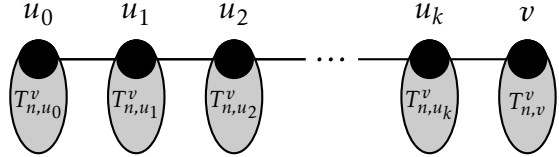

    Consider a Pólya urn with $k+2$ colors representing the subtrees attached to each of the vertices in the path from the root to vertex $v$, $T_{n,u_i}^v$ for $i\in\{0,1,\ldots,k+1\}$ (see figure \ref{subtreesPathFigure}). It is well-known that the vector of the sizes of the subtrees divided by $n$ converges almost surely to a  Dirichlet-distributed vector. Hence, for $n$ large enough, the product on the right-hand side of (\ref{vlowerbound}) converges to some constant, as each term $|T_{n,u_j}^v|$ in the sum converges to some constant, and hence each factor in the product converges. 

    Therefore, for $n$ large, the term $\prod_{i=1}^{k} \frac{\sum_{\ell=0}^{k+1-i}|T_{n,u_{\ell}}^v|}{\sum_{j=i}^{k+1}|T_{n,u_j}^v|}$ converges to some constant $\mu_w$. Consequently, the problem can be cast in the setting of Lemma \ref{diagonalRW}, as by Lemma \ref{lowerboundsubtreesize} there exists a random variable $\xi_w'$ such that $|(T_n,w)_{1\downarrow}| \ge n/\xi_w'$, and thus for each new vertex that attaches to vertex $u_k$, a random walk as in Lemma \ref{diagonalRW} is started. Almost surely---by Lemma \ref{diagonalRW} with $x = 1/(1+\mu_w)$---only finitely many children of $u_k$ can be more central than the root.

    For each vertex $v_k$ at distance $k$ from the root, 
    only finitely many of its descendants can become more central than the root almost surely. 
    By the induction hypothesis, for vertices at distance $k$, almost surely, only finitely many can ever become more central than the root. Moreover, since for a vertex at distance $k+1$ to be more central than the root, its parent must also be more central than the root, there are only finitely many vertices at distance $k+1$ that become more central than the root. Hence, by induction, the result holds almost surely for all $k\ge 1$.
\end{proof}

\begin{lemma}\label{finiteDistance}
    There exists a random variable $K>0$ with $\PROB(K<\infty)=1$ such that no vertex $v\in V^{(j)}$ for $j\ge K$  can ever become as central as the root. 
\end{lemma}
\begin{proof}
Define the event 
\begin{equation*}
    \mathcal{A} = 
\left\{  
\bigcap_{d \in \mathbb{N}} \bigcup_{v \in \N}
\bigcup_{n \geq v} 
\left[ 
\left[
\text{dist}(v,1) \ge d 
\right] 
\cap
\left[ 
\psi_{T_n}(v) \le \psi_{T_n}(\text{1})
\right] 
\right]  
\right\}~.
\end{equation*}
Hence, we aim to show that $\PROB(\mathcal{A})=0$. By persistence of the rumor center and the union bound, the probability can be upper-bounded as 
\begin{align*}
    \PROB(\mathcal{A}) \le \sum_{k=1}^{\infty} \PROB(\mathcal{A}\cap \{\text{vertex }k \text{ is the persistent centroid}\})~.
\end{align*}
Let $v^*$ be the persistent centroid. We know that the sizes of the subtrees obtained by removing the edges among the first $k$ vertices divided by $n$ converge almost surely to a Dirichlet-distributed random vector. Let $u$ be a vertex in the path from the root to the persistent centroid, excluding the endpoints. Then, an upper bound of the centrality of the root is
\begin{align*}
    \varphi_n(1) \le \left(\frac{n-|(T_n,u)_{1\downarrow}|}{|(T_n,u)_{1\downarrow}|} \right)^{\text{dist}_n(1,u)} \varphi_n(u)~.
\end{align*}
Similarly, for any vertex $v$ in the subtree of $u$, $T_{n,u}^{v^*}$, a lower bound of its centrality is
\begin{align*}
    \varphi_n(v) \ge \left(\frac{n-|T_{n,u}^{v^*}|}{|T_{n,u}^{v^*}|} \right)^{\text{dist}_n(v,u)} \varphi_n(u)~.
\end{align*}

Thus, if
\begin{align*}
\text{dist}_n(v,u) &> \text{dist}_n(1,u) \phi_n(u)~,
\end{align*}
where
\begin{align*}
\phi_n(u)
=
\frac{
\log \left(\frac{1-|(T_n,u)_{1\downarrow}|/n}{|(T_n,u)_{1\downarrow}|/n} \right)
}
{
\log\left( \frac{1-|T_{n,u}^{v^*}|/n}{|T_{n,u}^{v^*}|/n}\right)
}~,
\end{align*}
then $\varphi_n(v) > \varphi_n(1)$. 
 From standard Pólya urn convergence results, we have   
\begin{align*}
    \frac{|(T_n,u)_{1\downarrow}|}{n}\xrightarrow[\ n \to \infty\ ]{\mathrm{a.s.}} \gamma \quad\text{ and } \quad \frac{|T_{n,u}^{v^*}|}{n}\xrightarrow[\ n \to \infty\ ]{\mathrm{a.s.}} \beta_u
\end{align*}
for some random variables $\gamma$ and $\beta_u$, with $\gamma,\beta_u>0$. By the continuous mapping theorem, $\phi_n(u)$ converges almost surely to some constant $\phi(u)$, provided that $\gamma,\beta_u<1/2$, which is the case as we have assumed that $u$ is not the persistent centroid. Hence, there exists some almost surely finite time $N>0$ such that for all $n\ge N $ all vertices in the subtree $T_{n,u}^{v^*}$ at distance larger than $\phi(u) \text{dist}(1,u)$ to the root are less central than the root.

One can similarly consider a vertex $v$ that attaches to subtrees of vertices whose path to the root passes through the persistent centroid $v^*$. Let $u_v$ be the first vertex from the path of $v^*$ to $v$ that is a neighbor of $v^*$. Define $N_n^{\text{ch}}(v^*)$ as the set of children of $v^*$ in $T_n$. 
A lower bound of the centrality for such vertices is
\begin{align*}
    \varphi_n(v) 
    &\ge \left(\frac{n-|(T_n,1)_{u_v\downarrow}|}{|(T_n,1)_{u_v\downarrow}|} \right)^{\text{dist}_n(v,v^*)} \varphi_n(v^*) \\
    &\ge \left(\frac{n-\max_{k\in N_n^{\text{ch}}(v^*)}|(T_n,1)_{k\downarrow}|}{\max_{k\in N_n^{\text{ch}}(v^*)}|(T_n,1)_{k\downarrow}|} \right)^{\text{dist}_n(v,v^*)} \varphi_n(v^*)~.
\end{align*}
Taking logarithms, we note that $\varphi_n(v) > \varphi_n(1)$ if and only if
\begin{align*}
    \text{dist}_n(v,v^*) &> \text{dist}_n(1,v^*)
    \frac{
    \log \left(\frac{1-|(T_n,v^*)_{1\downarrow}|/n}{|(T_n,v^*)_{1\downarrow}|/n} \right)}
    {\log\left( \frac{1-\max_{k\in N^{\text{ch}}(v^*)}|(T_n,1)_{k\downarrow}|/n}{\max_{k\in N^{\text{ch}}(v^*)}|(T_n,1)_{k\downarrow}|/n}\right)
    }~. 
\end{align*}
The convergence of the subtree sizes implies that $\max_{k\in N_n^{\text{ch}}(v^*)}|(T_n,1)_{k\downarrow}|/n \xrightarrow[n\to\infty]{a.s.} D$, for some random variable $D$. If $D\ge 1/2$, then $v^*$ is not the persistent centroid, and therefore $D<1/2$. Hence, there exists an integer $K>0$ such that all vertices at a distance further than $K$ from the persistent centroid have a lower bound for the rumor centrality that exceeds the upper bound of the root's centrality. Thus, almost surely, there exists a finite time after which no vertex in the subtree $(T,1)_{v^*\downarrow}$ at distance larger than $K>0$ from the root becomes more central than the root. Therefore, for all $v^*$, $\PROB(\mathcal{A}\cap \{\text{vertex $v^*$ is the persistent centroid}\})=0$. 
\end{proof}

\begin{theorem}\label{rumorPersistence}
  The rumor rank  $R_n$  of the root is persistent.
\end{theorem}
\begin{proof}
By Lemmas \ref{persistenceDistance} and \ref{finiteDistance}, there exists an almost surely finite distance beyond which no vertex can ever surpass the root in centrality. Furthermore, for any fixed distance, almost surely only finitely many vertices ever become more central than the root. Consequently, the total set of vertices that are ever more central than the root is almost surely finite. Since Lemma \ref{finitechangeneighbors} establishes that any finite set of vertices possesses a persistent relative ranking, it follows that the rank of the root is persistent almost surely.
\end{proof}

\section{Closeness centrality}
\label{sec:closeness}

In this section, we explore properties of the rank $R_n$ of the root when vertices are ranked according to their
closeness centrality. Recall that the most central vertex coincides with the centroid, and therefore the properties
of $I_n$ are well understood. See Section \ref{sec:jordan}. In contrast, 
closeness centrality has not been explored in the context of root estimation.
Here we show that, perhaps surprisingly, $R_n$ is not tight, that is, closeness centrality is not good for
root estimation. Moreover, we show that $\EXP R_n =  \exp\left(\Omega\left(
(\log(n))^{1/3}\right) \right)$ and that $R_n$ is not persistent.

\subsection{Tightness, expected rank of the root, and root estimation}

\begin{theorem}\label{ClosenessRootEstimation}
Closeness centrality is not good for root estimation.
\end{theorem}

\begin{proof}
  We show that on an event of positive probability, the rank of the root diverges.
  Consider the event when the first child of the first child of the root, which we denote by $v_{2,1}$, has a large subtree so that $|(T_n,1)_{v_{2,1}\downarrow}| >\frac{3n}{4}$. This happens with a limiting probability
  $\PROB(U_1U_{1,1} > \frac{3}{4})$, where $U_1$ and $U_{1,1}$ are independent uniform $[0,1]$ random variables.
 On this event, consider the closeness centrality of any leaf $v_{2,1,\ell}$ attached to vertex $v_{2,1}$:
    \begin{equation*}
        \mathcal{C}_n(v_{2,1,\ell}) = \mathcal{C}_n(v_{2,1}) + 2 \left( \frac{n}{2}-1 \right)~,
    \end{equation*}
    whereas the closeness centrality of the root equals
    \begin{equation*}
        \mathcal{C}_n(1) = \mathcal{C}_n(v_{2,1}) + 2 \left(|(T_n,1)_{v_{2,1}\downarrow}| - |(T_n,v_{2,1})_{1 \downarrow}|  \right)~.
    \end{equation*}
    Therefore,
    \begin{equation*}
        \mathcal{C}_n(1) > \mathcal{C}_n(v_{2,1,\ell}) \iff  \left(|(T_n,1)_{v_{2,1}\downarrow}| - |(T_n,v_{2,1})_{1 \downarrow}|  \right) > \frac{n}{2}-1~.
    \end{equation*}
    This occurs when $|(T_n,1)_{v_{2,1}\downarrow}| >\frac{3n}{4}$. Any non-leaf child of vertex $v_{2,1}$ is more central than a leaf child. So all neighbors of vertex $v_{2,1}$ are more central than the root conditional on $|(T_n,1)_{v_{2,1}\downarrow}| >\frac{3n}{4}$. 
    Conditioned on this event, the degree of vertex $v_{2,1}$ is at least $(1-\epsilon)\log(3n/4)$, and with high probability, all of them are more central than the root.
\end{proof}

\begin{theorem}\label{ClosenessExpected}
Let $R_n$ be the closeness rank of the root. Then
    \begin{equation*}
        \EXP R_n = \exp \left(\Omega\left( \log (n)^{1/3} \right)\right)~.
    \end{equation*} 
\end{theorem}

\begin{proof} 
Consider a vertex $v$ at distance $k$ to the root. We consider the event that $|(T_n,1)_{v\downarrow}| >n(2k-1)/(2k)$. When this event occurs, it implies that all vertices in paths of length $k-1$  starting at vertex $v$ and within $(T_n,1)_{v\downarrow}$ are more central than the root.

Let $v_{1^{(k)}}$ be defined recursively as follows:  $v_{1^{(1)}}$ is the first child of the root, and vertex $v_{1^{(k+1)}}$ is the first child of the vertex $v_{1^{(k)}}$. We write $X_{k-1}^{1^{(k)}}$ for the number of paths of length $k-1$ starting at vertex $v_{1^{(k)}}$ and whose vertices are all within $(T,1)_{v_{1^{(k)}}}$. See Figure \ref{fig:closeness} for an illustration of a subtree, $(T,1)_{v_{1^{(3)}}}$.
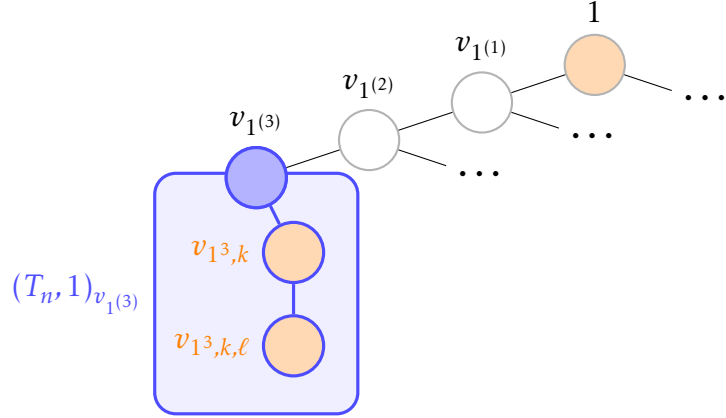
\begin{figure}[ht]
    \centering
    \begin{tikzpicture}[
  level distance=0.5cm,
  level 1/.style={sibling distance=3cm},
  level 2/.style={sibling distance=3cm},
  vertex/.style={circle, draw=gray!60, fill=white, minimum size=8mm, font=\small, thick},
  bigsubtree/.style={rounded corners=8pt, draw=blue!70, very thick, fill=blue!30, fill opacity=0.2, text opacity=1},
  smallsubtree/.style={rounded corners=8pt, draw=blue!40, very thick, fill=blue!10, fill opacity=0.1, text opacity=1},
]

\node[vertex,label={[text=orange!100]above: },fill=orange!30, fill opacity=1,label = 1] (root) {}
    child { node[vertex,label = $v_{1^{(1)}}$] (v1) {}
        child { node[vertex,label = $v_{1^{(2)}}$] (v12) {}
            child { node[vertex, fill= blue!30, fill opacity= 1, draw=blue!70, very thick,label = above : $v_{1^{(3)}}$] (v11) {} }
            child { node[font=\Large] (v13) {$\cdots$} }}
            child { node[font=\Large] (v122) {$\cdots$}}
    }
    child { node[font=\Large] (vdots) {$\cdots$} };
\node[bigsubtree, below =-0.5cm of v11, minimum width=2.7cm, minimum height=3.2cm, label={[text=blue!70]  left:{$(T_n,1)_{v_{1^{(3)}}}$ }}] (cloud) {};

\node[vertex, below right=0.4cm and -0.1cm of v11,,fill=orange!30, fill opacity=1,, draw=blue!70, very thick, label ={[text=orange!100] left: $v_{1^{3},k}$ }] (v111) {};
\draw[thick, draw=blue!70, very thick] (v11) -- (v111);
\node[vertex, below=0.4cm and 0cm of v111,,fill=orange!30, fill opacity=1,, draw=blue!70, very thick, label ={[text=orange!100]left: $v_{1^{3},k,\ell}$ }] (v1111) {};
\draw[thick, draw=blue!70, very thick] (v111) -- (v1111);
\node[vertex, fill=blue!30, fill opacity=1, draw=blue!70, very thick] at (v11) {};
\node[anchor=south east, font=\small] at (v11.north west) {};
\end{tikzpicture}

    \caption{Representation of the subtree of $v_{1^{(3)}}$ and paths of length $2$ within it.}
    \label{fig:closeness}
\end{figure}

We will show that the expected rank of the root is lower-bounded by
\begin{align}
    \EXP R_n &\ge \sum_{k=2}^n\EXP\left[X_{k-1}^{v_{1^{(k)}}} \bigg| |(T_n,1)_{v_{1^{(k)}}\downarrow}| \ge n\frac{2k-1}{2k} \right] \PROB\left(|(T_n,1)_{v_{1^{(k)}}\downarrow}| \ge n\frac{2k-1}{2k}\right) \nonumber\\
    &\ge \sum_{k=2}^{n/2}  \frac{\log ^k (n) - \gamma \log^{k-1} (n)}{k!} \PROB\left(|(T_n,1)_{v_{1^{(k)}}\downarrow}| \ge n\frac{2k-1}{2k}\right)~, \label{lowerboundRnCloseness}
\end{align}
where $\gamma > 0$ is a constant.

If $ | (T_n,1)_{v_{1^{(2)}}\downarrow}|>3n/4$ then, by the argument from the proof of Theorem \ref{ClosenessRootEstimation}, we know that any vertex attached to vertex $v_{1^{(2)}}$ is more central than the root. Therefore, the expected rank of the root is at least of order $\log(n)$. 

Next, we show that 
\[
\EXP\left[X_{k-1}^{v_{1^{(k)}}} \bigg| |(T_n,1)_{v_{1^{(k)}}\downarrow}| \ge n\frac{2k-1}{2k} \right]
\ge \frac{\log ^k (n)}{k!} +O\left(\log^{k-2}(n)\right)\]
holds for $k=3$.
It is straightforward to extend the argument to larger values of $k$. The centrality of vertex $v_{1^{(2)}}$ and that of the root can be written as 
\begin{align*}
    \mathcal{C}_n(v_{1^{(2)}}) &= \mathcal{C}_n(v_{1^{(3)}})+| (T_n,v_{1^{(2)}})_{v_{1^{(3)}}\downarrow}| - | (T_n,v_{1^{(3)}})_{v_{1^{(2)}}\downarrow}| 
\end{align*}
and 
\begin{align*}
     \mathcal{C}_n(1) &= \mathcal{C}_n(v_{1^{(3)}})+| (T_n,v_{1^{(2)}})_{v_{1^{(3)}}\downarrow}| - | (T_n,v_{1^{(3)}})_{v_{1^{(2)}}\downarrow}| + 2 \left(|(T_n,1)_{v_{1^{(2)}}\downarrow}| - |(T_n,v_{1^{(2)}})_{1 \downarrow}|  \right)~.
\end{align*}
If $ | (T_n,1)_{v_{1^{(3)}}\downarrow}|>5n/6$ then we have both 
$$
    |(T_n,v_{1^{(2)}})_{v_{1^{(3)}}\downarrow}| - | (T_n,v_{1^{(3)}})_{v_{1^{(2)}}\downarrow}| \ge \frac{2n}{3}
    $$
    and
    $$
    |(T_n,1)_{v_{1^{(2)}}\downarrow}| - |(T_n,v_{1^{(2)}})_{1 \downarrow}| \ge \frac{2n}{3}~.
$$
Hence, 
\begin{align*}
     \mathcal{C}_n(1) &\ge \mathcal{C}_n(v_{1^{(3)}})+ \frac{2n}{3} + 2 \frac{2n}{3} = \mathcal{C}_n(v_{1^{(3)}})+ 2n~.
\end{align*}
By the argument from the proof of Theorem \ref{ClosenessRootEstimation}, any vertex attached to vertex $v_{1^{(3)}}$  is more central than the root. Let $v_{1^{(3)},k}$ be a child of vertex $v_{1^{(3)}}$. Now consider any leaf $v_{1^{(3)},k,\ell}$ attached to $v_{1^{(3)},k}$ (see Figure \ref{fig:closeness}). The centrality of such vertices can be expressed as
\begin{align*}
    \mathcal{C}_n(v_{1^{(3)},k,\ell}) &= \mathcal{C}_n(v_{1^{(3)},k}) + n-2~.
\end{align*}    
Since $v_{1^{(3)},k,\ell}$ is a child of $v_{1^{(3)},k}$,
\begin{align*}
    \mathcal{C}_n(v_{1^{(3)},k}) &\le \mathcal{C}_n(v_{1^{(3)}}) +n-4 ~.
\end{align*}    
Hence, 
\begin{align*}
 \mathcal{C}_n(v_{1^{(3)},k,\ell}) \le \mathcal{C}_n(v_{1^{(3)}})+ 2n-6~.
\end{align*}
Therefore, if $ | (T_n,1)_{v_{1^{(3)}}\downarrow}|>5n/6$, all vertices in any path of length $2$ started at $v_{1^{(3)}}$ and within $(T_n,1)_{v_{1^{(3)}}\downarrow}$ are more central than the root. Hence, a lower bound on the expected rank of the root $R_n$ is the expected number of paths of length $2$ starting at vertex $v_{1^{(3)}}$ in $(T,1)_{v_{1^{(3)}}\downarrow}$ times the probability that $ | (T_n,1)_{v_{1^{(3)}}\downarrow}|>5n/6$. By definition, $X_2^{1^{(3)}}$ is the number of such paths.
\begin{align*}
    \EXP[X_2^{1^{(3)}}|&| (T_n,1)_{v_{1^{(3)}}\downarrow}|>5n/6] \ge \sum_{v=1}^{5n/6}\frac{1}{v} \sum_{k=v+1}^{5n/6}\frac{1}{k} = \frac{1}{2}\log^2(n) + O(\log n),
\end{align*}
since, conditioned on the subtree's size, the subtree itself evolves as a \textsc{urrt}. 
Similarly, one can show that for vertex $v_{1^{(k)}}$ at distance $k$ of the root, if $|(T_n,1)_{v_{1^{(k)}}\downarrow}|>n\frac{2k-1}{2k}$, then all vertices in any path of length $k-1$ starting at $v_{1^{(k)}}$ and within $(T_n,1)_{v_{1^{(k)}}\downarrow}$ are more central than the root. Hence, 
$$\EXP\left[X_{k-1}^{1^{(k)}} \bigg| |(T_n,1)_{v_{1^{(k)}}}| \ge n\frac{2k-1}{2k} \right] \ge \frac{\log^{k-1}(n)}{(k-1)!} +O\left(\log^{k-2}(n)\right)~.
$$
To lower-bound $\PROB\left(|(T_n,1)_{v_{1^{(k)}}}| \ge n \frac{2k-1}{2k}\right)$ in (\ref{lowerboundRnCloseness}), note that
\begin{align}
    \lim_{n\to\infty}\PROB\left(\frac{|(T_n,1)_{v_{1^{(k)}}}|}{n} \ge \frac{2k-1}{2k}\right)  
    &= \PROB\left(\prod_{i=1}^k U_i \ge \frac{2k-1}{2k}\right) \nonumber\\
     &= \PROB\left(\Gamma(k,1) \le \log\left(\frac{2k}{2k-1}\right)\right) \nonumber\\
     & \ge \frac{ \log(2k/(2k-1))^k /2} { k! } \nonumber\\
     & \ge \frac{ (1/(2k))^k /2} { k! } \nonumber\\
     & \ge \frac{ (e/2)^k} { 2 k^{2k} },\label{lowerPlargesubtree}
\end{align}
where the $U_i$ are independent uniform random variables on $[0,1]$ and
$\Gamma(k,1)$ is a gamma-distributed random variable,
and we used the inequality
$$
\PROB\left(\Gamma(k,1) \le x \right)
\ge \frac{x^k e^{-x}}{k!} .
$$
Combining the lower bound in (\ref{lowerPlargesubtree}) with \eqref{lowerboundRnCloseness} yields
\begin{align*}
   \EXP R_n = \exp \left( \Omega \left((\log n)^{1/3} \right) \right).
\end{align*}
\end{proof}


\subsection{Persistence}

\begin{theorem}\label{ClosenessPersistence}
  The closeness rank of the root $R_n$ is not persistent.
\end{theorem}

\begin{proof}
  We show that, for any fixed positive integer $k$, with probability one, only finitely many of the events
  $\{R_{n}\le k\} \cap \{|(T_n,1)_{11\downarrow}|>3n/4\}$ occur.
\begin{align*}
\PROB(R_{n}\le k \ \text{and} \  |(T_n,1)_{v_{1^{(2)}}\downarrow}|>3n/4) & \le\PROB\left( R_{n}\le k \Big{|} |(T_n,1)_{v_{1^{(2)}}\downarrow}|>3n/4\right) \\
    &\le \PROB\left(\text{deg}_n(v_{1^{(2)}})\le k \Big{|} |(T_n,1)_{v_{1^{(2)}}\downarrow}|>3n/4\right) \\
    &\le \PROB\left(\text{deg}_n(v_{1^{(2)}})\le k \Big{|} |(T_n,1)_{v_{1^{(2)}}\downarrow}|=3n/4\right) .
\end{align*}
Conditionally on the event $ |(T_n,1)_{v_{1^{(2)}}\downarrow}|=3n/4$, the subtree $(T_n,1)_{v_{1^{(2)}}\downarrow}$ is a random recursive tree rooted at vertex $v_{1^{(2)}}$ with $3n/4$ vertices. Hence, the distribution of the degree of vertex $v_{1^{(2)}}$ is a sum of independent Bernoulli random variables. Therefore, considering the subsequence $\{2^n\}_{n=1}^{\infty}$ and using a fourth central moment inequality, we have
\begin{align*}
    \sum_{n=1}^{\infty }&\PROB\left(R_{2^n}\le k \Big| | (T_{2^n},1)_{v_{1^{(2)}}\downarrow}|>\frac{3}{4}2^n\right)\le  \sum_{n=1}^{\infty } \PROB \left(\sum _{j=1}^{3 \cdot 2^n/4} Ber(1/j) \le k  \right) < \infty.
\end{align*}

Hence, by the Borel-Cantelli lemma, along the subsequence $\{2^n\}_{n=1}^{\infty}$ almost surely only finitely many of the events $\{R_{n}\le k\} \cap \{|(T_n,1)_{v_{1^{(2)}}\downarrow}|>3n/4\}$ occur. 

Since
\begin{equation*}
\inf_{n>0} \PROB\left(
|(T_{n},1)_{v_{1^{(2)}}\downarrow}|>3n/4
\right) > 0~,
\end{equation*}
and $|(T_n,1)_{v_{1^{(2)}}\downarrow}|/n$ is a bounded sub-martingale, we have that 
\begin{equation*}
    \PROB \left(\cup_N \cap_{n\ge N} \left[ (T_n,1)_{v_{1^{(2)}}\downarrow}|\ge 3n/4 
    \right] \right)>0~.
\end{equation*}
So, the rank of the root is not persistent.
\end{proof}

\printbibliography


\end{document}